\documentclass{birkjour}
\usepackage{hyperref}
\usepackage{amsmath}
\usepackage{amssymb}
\usepackage{amsfonts}
\usepackage{amsthm}
\usepackage{tikz,tikz-cd,float}
\usepackage{latexsym}
\usepackage{esint}
\usepackage{mathtools}
\usepackage{mathrsfs}
\usepackage{graphicx}
\usepackage{bbm}
\usepackage{color}
\usepackage{bm}

\newtheorem{theorem}{Theorem}[section]
\newtheorem{lemma}[theorem]{Lemma}
\newtheorem{proposition}[theorem]{Proposition}
\newtheorem{corollary}[theorem]{Corollary}

\theoremstyle{definition}
\newtheorem{definition}[theorem]{Definition}

\theoremstyle{remark}
\newtheorem{remark}[theorem]{Remark}

\numberwithin{equation}{section}

\usepackage{stmaryrd}
\newcommand{\bx}{\mathbf{x}}
\newcommand{\by}{\mathbf{y}}

\newcommand{\bk}{\mathbf{k}}
\newcommand{\bj}{\mathbf{j}}
\newcommand{\bl}{\bm{\ell}}
\newcommand{\bff}{\mathbf{f}}
\newcommand{\bg}{\mathbf{g}}
\newcommand{\bu}{\mathbf{u}}

\newcommand{\bv}{\mathbf{v}}
\newcommand{\bw}{\mathbf{w}}
\newcommand{\bn}{\mathbf{n}}

\newcommand{\dd}{\,\mathrm{d}}

\newcommand{\dx}{\, \mathrm{d} \mathbf{x}}

\newcommand{\dt}{\, \mathrm{d}t}
\newcommand{\ds}{\, \mathrm{d}\sigma}

\newcommand{\divx}{\, \mathrm{div}}

\begin{document}

%-------------------------------------------------------------------------
% editorial commands: to be inserted by the editorial office
%
%\firstpage{1} \volume{228} \Copyrightyear{2004} \DOI{003-0001}
%\seriesextraline{This is the Concrete Title of this Book\br H.E. R and S.T.C. W, Eds.}
%\firstpage{1}

\title[Spatial analyticity and exponential decay of Fourier modes]{Spatial analyticity and exponential decay of Fourier modes for the stochastic Navier-Stokes equation}

%----------Author 1
\author{Dan Crisan}

\address{Department of Mathematics\\ Imperial College London\\ SW7 2AZ, UK.}
\email{d.crisan@imperial.ac.uk}

%\author[Birkh\"auser]{Birkh\"{a}user Publishing Ltd.}
%
%\address{%
%Viaduktstr. 42\\
%P.O. Box 133\\
%CH 4010 Basel\\
%Switzerland}
%
%\email{info@birkhauser.ch}

\thanks{We would sincerely like to thank the referees for the many useful suggestions and corrections in the making of this manuscript.}
%----------Author 2
\author{Prince Romeo Mensah*}

\address{Department of Mathematics\\ Imperial College London\\ SW7 2AZ, UK.}
\address{AND}
\address{Institute of Mathematics\\ TU Clausthal\\ Erzstraße 1, 38678 \\Clausthal-Zellerfeld,\\ Germany}
\email{prince.romeo.mensah@tu-clausthal.de}


\date{\today}

\begin{abstract}
{We construct a local in time spatially real-analytic} solution to the 2D and 3D stochastic Navier--Stokes equation driven by a  spatially real-analytic multiplicative and transport noise but emanating from an initial condition that is only required to have bounded enstrophy. 
Under the condition that the solution is global in time,  we also establish the exponential decay of the finite-dimensional Galerkin approximation, { with respect to its maximum wavenumber,} to the strong pathwise solution of the stochastic Navier--Stokes equation. This decay is uniform in time, uniform with respect to the initial enstropy, and {uniform in the noise coefficients.}
\end{abstract}

%%% ----------------------------------------------------------------------
\maketitle
%%% ----------------------------------------------------------------------
%\tableofcontents
\section{Introduction}
\noindent The Navier--Stokes equation is a widely used model to describe the evolution of incompressible viscous fluids. It represents a balance law between the forces present in the fluid and the acceleration of the fluid. Classically, the main forces taken into account will consist of the deterministic conservative pressure, non-conservative viscous term, and perhaps, all other external deterministic forces lumped up together.  However, to take uncertainties in the fluid into account, it is desirable to incorporate noise into the system. This noise may be understood as a random force that is a function of the underlining deterministic unknowns.
{Another form of noise is the transport noise of Stratonovich-type, that has the desirable property} of preserving some of the physical laws satisfied by the deterministic system. 
\\
In this work, we are interested in {the regularity of solutions} to the 2D and 3D stochastic Navier--Stokes equation for incompressible fluids. In particular, we explore a specific Gevrey class regularity that renders the solution spatially real-analytic. The analysis of Gevrey class regularity for fluid-dynamic systems has a long history spanning many decades. The fundamental work by Foias and Temam \cite{foias1989gevrey}, in particular, bridged the gap between the functional analysis of general differential equations and the study of regularity for solutions to fluid dynamic equations. Since the work done by Foias and Temam, the treatment of Gevrey class regularity for deterministic fluid dynamic systems has gained considerable attention. We cannot give an exhaustive review but some fundamental works include \cite{bae2012analyticity, cao2000gevrey, ferrari1998gevrey, foias2001navier, kukavica2009radius, levermore1997analyticity, paicu2011analyticity}. 
%See the manuscript \cite{foias2001navier} for further details.
Unfortunately, very few works exist on the Gevrey class regularity for stochastic systems. From an analytic point of view, the only work we have found is by Mattingly \cite{mattingly2002dissipative} who analysed the Navier--Stokes equation with a white-in-time stochastic forcing. An analytic solution is derived from analytic data and discussions on stationary measures and ergodicity are also presented. From the numerical analysis point of view, a few more results are present \cite{kukavica2018galerkin, lord2004numerical}. In \cite{kukavica2018galerkin}, the authors show the convergence of the Galerkin approximation to the strong solution of the 2D Navier--Stokes equation  defined on a bounded domain with Dirichlet boundary condition and that is driven by a multiplicative noise. However, the rate of convergence is not given and { the convergence holds for the expectation of a logarithmic function of the norm of the difference of the two solutions, rather than for the norm itself.} In \cite{lord2004numerical}, however, the authors give an inverse polynomial rate of convergence for general (viscous) stochastic PDEs driven by additive noise whose coefficients are analytic.
\\
In this work, we analyze a generalized Navier--Stokes equation on a 2D and 3D torus that accounts for all the deterministic and random forces discussed above. We explore its Gevrey class regularity that results in the unknowns becoming spatially real-analytic. In particular, { for data} consisting of just a bounded initial enstrophy and spatially real-analytic forcing terms, { we construct a local in time spatially real-analytic} solution to the Navier--Stokes equation. This is a regularization result in the sense that we obtain an analytic solution from an initial condition that is far less regular as observed in  many deterministic dissipative systems \cite{biswas, cao2000gevrey, ferrari1998gevrey, foias1989gevrey, paicu2011analyticity}. However, a caveat is that we need more regularity for the forced data of stochastic transport type than the anticipated regularity of the solution to be constructed. As such, although the solution becomes regularized with respect to its initial condition, {  the trade-off is that we need more regular transport noise coefficients, in particular we assume they are constant vectors in the time variable.
} However, the assumptions are optimal for noise with linear growth and Lipschitz coefficients (like additive and multiplicative noise) where we only require the same regularity for the noise coefficient as that of the anticipated solution. To our knowledge, this regularity result is a first for the stochastic Navier--Stokes equation. 
\\
Our second main result involves the analysis of the rate of decay of the finite-dimensional solution used in the approximation of the constructed solution. Here, we assume that the constructed solution is global in time. We
measure the decay rate in the space of enstrophy functionals and show that the rate of decay is exponential as a function of its Fourier modes. We refer to Theorem \ref{thm:decay} for the exact statement. As a consequence, in particular, we can infer that irrespective of the addition of noise, each Fourier coefficient of our solution decays { faster than exponentially} (see the first estimate in \eqref{eq:betterThanExp}), uniformly in time, uniform with  respect to the initial enstrophy and { uniformly in} the noise coefficients, with the same rate as we will expect of the corresponding deterministic system \cite{doering1995exponential, foias2001navier}. Also, see \cite{duan} for a related result for the generalized Ginzburg-Landau equation. Subsequently, the difference between the continuum and discrete velocities in the spatial $H^1$-norm decays exactly exponentially fast to zero.
\subsection{Plan}
In order to improve readability, we begin with some preliminaries covered in Section \ref{sec:prelim}, where we introduce the various notations and functional set-up needed to present and { prove our main results.} We then state the precise equation we wish to study in Section \ref{sec:MainEquation} and give the assumptions under which the equation is analysed. We also make precise in this section the various concepts of a solution and finally, state our main results. In order to avoid long computations in our proofs, we present in  Section \ref{sec:preparation} some key results that will be used in the proof of our main results. Next, we devote the entirety of Section \ref{sec:GevReg} to the proof of our first main result concerning the construction of a unique  spatially real-analytic solution of the  stochastic Navier--Stokes equation. 
Here, the main tool in our construction is the application of the extension of the Cauchy method \cite[Lemma 5.1]{glatt2009strong} to the Gevrey classes taking into account the additional transport noise and its Stratonovich-to-It\^o corrector. Finally, we use Section \ref{sec:decay} to study the asymptotic behaviour of the finite-dimensional Galerkin approximations which is the subject of our second main result, Theorem \ref{thm:decay}.

\section{Preliminaries}
\label{sec:prelim}
\subsection{Notation}
We consider a spacetime cylinder consisting of  spatial points $ \mathbf{x}=(x_1,\ldots, x_d) $ on the  torus $\mathbb{T}^d=(\mathbb{R}/2\pi\mathbb{Z})^d$, $d=2,3$ with periodic boundary condition  and a time variable  $t\in [0,T]$ where $T>0$ is fixed and arbitrary. For functions $F$ and $G$, we write $F \lesssim G$  if there exists  a generic constant $c>0$  such that $F \leq c\,G$.
We also write $F \lesssim_p G$ if the  constant  $c(p)>0$ depends on a variable $p$. If $F \lesssim G$ and $G\lesssim F$ both hold (respectively,  $F \lesssim_p G$ and $G\lesssim_p F$), we use the notation $F\sim G$ (respectively, $F\sim_p G$).
The symbol $\vert \cdot \vert$ may be used in four different context. For a scalar function $f\in \mathbb{R}$, $\vert f\vert$ denotes the absolute value of $f$. For a vector $\bff\in \mathbb{R}^d$, $\vert \bff \vert$ denotes the Euclidean norm of $\bff$. For a square matrix $\mathbb{F}\in \mathbb{R}^{d\times d}$, $\vert \mathbb{F} \vert$ shall denote the Frobenius norm $\sqrt{\mathrm{trace}(\mathbb{F}^T\mathbb{F})}$. Lastly, { if $S\subseteq  \mathbb{R}^d$ is  a measurable subset}, then $\vert S \vert$ is the $d$-dimensional Lebesgue measure of $S$.
\\ \\
Throughout this paper, we fix a stochastic basis $(\Omega, \mathcal{F}, (\mathcal{F}_t)_{t\geq 0}, \mathbb{P})$  with the complete right-continuous filtration $(\mathcal{F}_t)_{t\geq 0}$. In addition, $\mathbb{E}(\cdot)=\int_\Omega (\cdot) \dd \mathbb{P}$ is the expectation of the argument denoted by the dot. Let $(W_t^k)_{k\geq1}$ be a family of one-dimensional $(\mathcal{F}_t)$-adapted  Wiener/Brownian processes.
 A \textit{stopping time} $\tau$ is a random variable $\tau :\Omega \rightarrow [0,T]$ such that $\{\omega \,:\, \tau(\omega) \leq t\} \in \mathcal{F}_t$ for each $t\in[0,T]$. 
We denote by $L^p_{\mathcal{F}_t}(\Omega;H)$, where $p\in[1,\infty]$, the space of equivalence class of random variables $f:\Omega \rightarrow H$ that are $\mathcal{F}_t$-measurable and for which { $\Vert f \Vert_H$ has} finite $p$-moments.

\subsection{The functional setting}
Let $C^\infty_{\divx}(\mathbb{T}^d)$ denote the space of all divergence-free smooth periodic functions. For $r\geq0$, let the Hilbert space $(H^r(\mathbb{T}^d), \Vert \cdot\Vert_{H^r})$ represent the $L^2$-homogeneous Sobolev space of mean-free functions  defined as
\begin{equation}
\begin{aligned}
\label{hsSpacehomoXX}
 H^r(\mathbb{T}^d)=\Bigg\{
&\bff(\bx)= \sum_{\bk\in  \mathbb{Z}^d}\hat{\bff}_{\bk} e^{i\bk\cdot\bx}
\in
L^2(\mathbb{T}^d)
\, \bigg\vert\,
\hat{\bff}_{\bk}
=\langle\bff(\bx), e^{i\bk \cdot\bx}\rangle_{L^2} 
\\&=\int_{\mathbb{T}^d} e^{-i\bk \cdot\bx}\bff(\bx) \dx \in \mathbb{C}^d,\,
  \overline{\hat{\bff}}_{\bk}= \hat{\bff}_{-\bk}, \quad \hat{\bff}_{\bm{0}}=0, \quad 
\\&
\Vert \bff \Vert_{ H^r} =\bigg( \sum_{\bk\in  \mathbb{Z}^d}  \vert \bk\vert^{2r} \vert\hat{\bff}_{\bk}\vert^2 \bigg)^\frac{1}{2}<\infty
\Bigg\}.
\end{aligned}
\end{equation}
Here, the $\hat{\bff}_{\bk}$s are   the Fourier coefficients of the function $\bff\in L^2(\mathbb{T}^d)$. 
\begin{remark}
Note the $\hat{\bff}_{\bm{0}}=0$ in  \eqref{hsSpacehomoXX} automatically renders elements in these spaces mean-free, i.e., $\int_{\mathbb{T}^d}\bff(\bx)\dx=0$ for any $\bff\in  H^r(\mathbb{T}^d)$. 
%The converse is also true.
Also, note that the underlining $L^2$-inner product is given by  $\langle \bff,\bg\rangle=\frac{1}{(2\pi)^d}\int_{\mathbb{T}^d}\bff(x)\cdot\overline{\bg}(x)\dx$.
\end{remark}
\noindent With \eqref{hsSpacehomoXX} in hand, we now define  $(\mathbb{H}^r(\mathbb{T}^d), \Vert \cdot\Vert_{H^r})$ as the subclass of $(H^r(\mathbb{T}^d), \Vert \cdot\Vert_{H^r})$ satisfying
\begin{equation}
\begin{aligned}
\label{hsSpacehomoYY}
 \mathbb{H}^r(\mathbb{T}^d)=\bigg\{
\bff(\bx)&= \sum_{\bk\in  \mathbb{Z}^d}\hat{\bff}_{\bk} e^{i\bk\cdot\bx}
\in
H^r(\mathbb{T}^d)
\, \big\vert\,
 \hat{\bff}_\bk\cdot \bk=0
\bigg\}.
\end{aligned}
\end{equation}
\begin{remark}
The condition $\hat{\bff}_\bk\cdot \bk=0$ corresponds to the divergence-free condition $\divx \bff=0$ in frequency space.
\end{remark}
\noindent { The space $ {H}^2(\mathbb{T}^d)$ is the domain of  the   positive, self-adjoint  Stokes operator
\begin{align*}
A= P(-\Delta)
%=-\Delta
, \qquad A:D(A)\rightarrow   \mathbb{H}^0(\mathbb{T}^d)=:\mathbb{L}^2(\mathbb{T}^d);
\end{align*}
the domain $D(A)$ is compactly embedded in ${L}^2(\mathbb{T}^d)$ and where $ P = I- \nabla \Delta^{-1}\mathrm{div}$ is the Leray projector. 
It is well-known that $P$ commutes with the derivative operators and it can be restricted to a bounded linear operator from  $H^r$ to $\mathbb{H}^r$,} see \cite{ lions1996mathematical, robinson, rockner2012stochastic}. Also, $A$ has a compact inverse $A^{-1}$ and
possesses a nondecreasing sequence  $\{\lambda_j\}_{j\in \mathbb{N}}$ of strictly positive eigenvalues approaching infinity as $j\rightarrow\infty$ with associated orthonormal basis $\{\mathbf{e}_j\}_{j\in \mathbb{N}}$  in $\mathbb{L}^2(\mathbb{T}^d)$.
By using the properties $\overline{\hat{\bff}}_{\bk}= \hat{\bff}_{-\bk}$ and $ \hat{\bff}_\bk\cdot \bk=0$, explicitly, we obtain for each $\bk\in  \mathbb{Z}^d$, 
\begin{align*}
\lambda_\bk=\vert \bk\vert^2, \qquad \mathbf{e}_\bk=
\mathbf{a}_\bk e^{i\bk\cdot\bx}+
\overline{\mathbf{a}}_\bk e^{-i\bk\cdot\bx}
\end{align*}
where for each $\bk$, $\mathbf{a}_\bk$ is a complex vector  in $\mathbb{C}^d$ with the property that $\mathbf{a}_\bk \cdot \bk=0$ and $\overline{\mathbf{a}}_\bk= \mathbf{a}_{-\bk}$.
{ The pair of sequences} $\{(\lambda_j,\mathbf{e}_j)\}_{j\in \mathbb{N}}$ then corresponds to the rearrangement of $\{(\lambda_\bk,\mathbf{e}_\bk)\}_{\bk\in  \mathbb{Z}^d}$ in nondecreasing order of $\lambda_\bk$.
\\
Next, by using the definition of the Stokes operator $A$, for any $r\in \mathbb{R}$, we can define its fractional power $ A^r$ as the mapping
\begin{align*}
\bff(\bx)= \sum_{\bk\in  \mathbb{Z}^d}\hat{\bff}_{\bk} e^{i\bk\cdot\bx}
\qquad \mapsto \qquad
%\Lambda^{2r}\bff:= 
A^r\bff
&=
\sum_{\bk\in  \mathbb{Z}^d}  \vert \bk\vert^{2r}  \hat{\bff}_{\bk} e^{i\bk\cdot\bx}
%= \sum_{\bk\in  \mathbb{Z}^d}  \lambda_{\bk}^r  \hat{\bff}_{\bk} e^{i\bk\cdot\bx}
\\&= \sum_{\bk\in  \mathbb{Z}^d}  \lambda_{\bk}^r  \langle\bff(\bx), e^{i\bk \cdot\bx}\rangle_{L^2} e^{i\bk\cdot\bx}
\end{align*}
with $A^0=I$ so that
\begin{align*} 
\Vert A^r\bff\Vert_{L^2}^2= \sum_{\bk\in  \mathbb{Z}^d}  \vert \bk\vert^{4r} \vert \hat{\bff}_{\bk} \vert^2
= \sum_{\bk\in  \mathbb{Z}^d}  \lambda_{\bk}^{2r}\vert  \hat{\bff}_{\bk} \vert^2
=
\Vert \bff \Vert_{ H^{2r}}^2.
\end{align*}
We let $P^N:L^2(\mathbb{T}^d) \rightarrow H_N$ be the $L^2$-orthogonal projection onto 
 $H_N=\text{span}\{e^{i\bk\cdot \bx} \, \text{ with }\,   \vert\bk \vert \leq N\}$. Note that  by using the definition of $P^N$, we directly obtain these  inequalities for trading Fourier modes for regularity
 \begin{equation}
\begin{aligned}
\label{poinare}
&\Vert P^N \bff\Vert_{ H^{s}}^2 \leq N^{2(s-r)}\Vert P^N \bff \Vert_{ H^{r}}^2,  \qquad
\\&\Vert (I-P^N) \bff\Vert_{ H^{r}}^2 \leq N^{2(r-s)} \Vert (I-P^N) \bff \Vert_{ H^{s}}^2
\end{aligned}
 \end{equation}
for $0\leq r\leq s$ and for all $N\in \mathbb{N}$.
Also, the continuity property
\begin{align}
\label{continuityProperty}
\Vert P^N \bff\Vert_{ H^{r}}^2 \lesssim \Vert\bff \Vert_{ H^{r}}^2,  \qquad\qquad
\Vert (I-P^N) \bff\Vert_{ H^{r}}^2 \lesssim \Vert \bff \Vert_{ H^{r}}^2
\end{align}
holds uniformly in $N\in \mathbb{N}$ for all $r\in \mathbb{R}$.
\\
{  With \eqref{hsSpacehomoYY} in hand, we can recall the so-called Gevrey class spaces. For $s>0$ and $r\geq 0$, we say that $\bff $ is of Gevrey class $s$ ( i.e. $\bff\in G^s(\mathbb{T}^d)$) if and only if { there exists $\varphi> 0$} such that
\begin{align}
\bff\in
D(e^{\varphi  A^{1/2s}} : \mathbb{H}^r(\mathbb{T}^d)) \equiv \big\{ \bff(\bx)\in \mathbb{H}^r(\mathbb{T}^d) \, : \, \Vert e^{\varphi  A^{1/2s}} \bff(\bx)  \Vert_{H^r}<\infty \big \}.
\end{align}
%\noindent With \eqref{hsSpacehomoYY} in hand, for a given $s>0$ and $r,\varphi\geq 0$,
%we can now define the Gevrey class-$(s,r,\varphi)$ space $G^{s,r}_{\varphi }(\mathbb{T}^d)$ as the domain of $(e^{\varphi  A^{1/2s}} : \mathbb{H}^r(\mathbb{T}^d))$, denoted as
%$D(e^{\varphi  A^{1/2s}} : \mathbb{H}^r(\mathbb{T}^d))$, and
%defined in terms of the Sobolev spaces $\mathbb{H}^r(\mathbb{T}^d)$ as follows
%\begin{align}
%D(e^{\varphi  A^{1/2s}} : \mathbb{H}^r(\mathbb{T}^d)) \equiv \big\{ \bff(\bx)\in \mathbb{H}^r(\mathbb{T}^d) \, : \, \Vert e^{\varphi  A^{1/2s}} \bff(\bx)  \Vert_{H^r}<\infty \big \}
%\end{align}
see \cite{levermore1997analyticity}. 
In the Gevrey spaces, we refer to $r$ as the Sobolev corrector and $\varphi$ is the radius of  analyticity or width of  analyticity. To account for variations in the radius of  analyticity, we consider
the parameter $\varphi$ as a function $\varphi=\varphi(t)$ of time, see \cite[Page 71]{foias2001navier}.
}
\begin{remark}
When we don't need to use divergence-free  vector fields, we can define the Gevrey space in terms of \eqref{hsSpacehomoXX} rather than \eqref{hsSpacehomoYY} and still retain the same properties above.
\end{remark}

\section{The main results}
\label{sec:MainEquation}
{ 
\noindent With our preparation in hand, we can now present the system under study. Here, we are interested in a velocity field $\bu:\Omega \times [0,T] \times \mathbb{T}^d\rightarrow \mathbb{R}^d$ and a pressure  $p:\Omega \times [0,T] \times \mathbb{T}^d\rightarrow \mathbb{R}$
that satisfy 
\begin{equation}
\begin{aligned}
\label{eq:mainLU00}
    &\dd \bu  +[  (\bu\cdot\nabla) \bu - \nu  \Delta \bu   ]\dt+\dd \nabla p = \sum_{k\geq1} \big[\bg_k(\bu)-(\bm{\xi}_k \cdot\nabla )\bu \circ \big] \dd W_t^k,
    \\
    &\divx \bu=0,\\
    &\bu(\bx,0)= \bu_0(\bx).
\end{aligned}
\end{equation}
Here, $\nu>0$ is the viscosity coefficient and the symbol $\circ$ means that the stochastic integral is understood in the Stratonovich sense. Details on the body forces $(\bg_k)_{k\geq1}$ and the coefficients of the noise advection $(\bm{\xi}_k)_{k\geq1}$ will be given in the next section.
\\
The system is a generalisation of the deterministic Navier--Stokes equation that takes into account several random phenomena in the fluids. In particular, this may cover additive noise, multiplicative noise and transport noise of Stratonovich type, and a first rigorous analysis can be found in \cite{mikulevicius2004stochastic}.
\\
When $\bm{\xi}_k\equiv0$, the resulting stochastic partial differential equation have been studied intensively by many authors including \cite{bensoussan, brzeniak, capinski, flandoli1995martingale}.
\\
On the other hand, when $\bg_k(\bu)\equiv0$, the system \eqref{eq:mainLU00} becomes a so-called SALT (Stochastic Advection by Lie Transport) model introduced by Holm \cite{holm1} but adjusted for viscous effects as studied in \cite{flandoli2021high} in vorticity form. The (inviscid) SALT models are derived from a stochastic variational principle and their solutions follow the flow of a stochastic vector field. The viscous SALT variant above are, however, closely related to the LU (Location Uncertainty) models introduced by M\'emin \cite{memin} that relies on a decomposition of the velocity fields into a random non-differentiable part and a differentiable deterministic component. In the specific case of the Stochastic Navier--Stokes under LU,  one  considers a random vector fields $\bm{\sigma}:\Omega \times [0,T] \times \mathbb{T}^d\rightarrow \mathbb{R}^d$ and define an elliptic  \textit{variance tensor} $\mathbb{A}=(A_{ij})_{i,j=1}^d\in \mathbb{R}^{d\times d}$ as $\mathbb{A}(\bx,t)\delta(t-s)\dt= \mathbb{E}[(\bm{\sigma}(\bx,t)\dd \mathbf{W}_t)(\bm{\sigma}(\bx,s)\dd \mathbf{W}_s)^T] $ where $\mathbf{W}_t=(W_t^k)_{k\geq1}$. The system of equation is then given by 
\begin{equation}
\begin{aligned}
\label{eq:mainLU}
    &\partial_t  \bu  + (\bu\cdot\nabla) \bu
    -
    \frac{1}{2}  (\divx\mathbb{A}\cdot\nabla)\bu    
     - \frac{1}{2}\divx(\mathbb{A}\nabla \bu)  =\nu \Delta \bu - \nabla p,
  \\
    &\nabla \dd \hat{p}_t=  \nu \Delta (\bm{\sigma} \dd \mathbf{W}_t)  -(\bm{\sigma} \cdot \nabla) \bu 
    \dd \mathbf{W}_t,
    \\
   & \divx\bu=0,
    \qquad \divx(\bm{\sigma} \dd \mathbf{W}_t)=0, \qquad \divx\divx\mathbb{A}=0,\\
    &\bu(\bx,0)= \bu_0(\bx)
\end{aligned}
\end{equation}
where $p(\bx,t)$ denotes the large-scale pressure contribution and $\hat{p}_t$ is a zero-mean turbulent pressure related to the small-scale velocity component.
\\
From the analytic point of view, \eqref{eq:mainLU} is strongly related to \eqref{eq:mainLU00} (with $\bg_k(\bu)\equiv0$) in the sense that both system preserves the deterministic energy since the noise terms and the variance tensor terms cancels out when deriving an identity for $\Vert\bu\Vert_{L^2}^2$, assuming that $\bu$ is smooth enough. 
\\
Returning to \eqref{eq:mainLU00}, it is convenient to eliminate the pressure in \eqref{eq:mainLU00} and only ask to find the velocity field. This is done by applying Leray's projection $ P = I- \nabla \Delta^{-1}\mathrm{div}$ to the first equation in \eqref{eq:mainLU00}. Doing so and converting the Stratonovich integral to an It\^o integral results in
\begin{equation}
\begin{aligned}
\label{eq:mainLU}
    &\dd \bu  +[ P ((\bu\cdot\nabla) \bu) - \nu  \Delta \bu   ]\dt =
    \frac{1}{2}\sum_{k\geq1} P((\bm{\xi}_k \cdot\nabla )P(\bm{\xi}_k \cdot\nabla )\bu)  \dt
    \\&\qquad\qquad\qquad\qquad\qquad\qquad
    + \sum_{k\geq1}P\big[\bg_k(\bu)-((\bm{\xi}_k \cdot\nabla )\bu)  \big] \dd W_t^k,
 \\&
 \divx \bu=0,
 \\&\bu(\bx,0)= \bu_0(\bx).
\end{aligned}
\end{equation}
Note that by using the incompressibility condition $\divx \bu=0$, one can recover the pressure from \eqref{eq:mainLU00}$_1$ by solving
\begin{align*}
\dd p =(- \Delta)^{-1}\divx((\bu\cdot\nabla) \bu)\dt- \sum_{k\geq1} (- \Delta)^{-1}\divx\big[\bg_k(\bu)-(\bm{\xi}_k \cdot\nabla )\bu \circ \big] \dd W_t^k.
\end{align*}
Also note that if we consider the orthogonal complement $Q$ of $P$, then for incompressible vector fields $\bu$ and $\bm{\xi}_k$, we have that
\begin{align*}
\big\langle P((\bm{\xi}_k \cdot\nabla )P(\bm{\xi}_k \cdot\nabla )\bu) \, ,\, \bu\rangle
&=
\langle (\bm{\xi}_k \cdot\nabla )P(\bm{\xi}_k \cdot\nabla )\bu \, ,\, \bu\rangle
\\
&=-
\langle P(\bm{\xi}_k \cdot\nabla )\bu \, ,\, (\bm{\xi}_k \cdot\nabla )\bu\rangle
\\
&=-
\langle P(\bm{\xi}_k \cdot\nabla )\bu \, ,\, P(\bm{\xi}_k \cdot\nabla )\bu\rangle
\\
&\qquad- 
\langle P(\bm{\xi}_k \cdot\nabla )\bu \, ,\, Q(\bm{\xi}_k \cdot\nabla )\bu\rangle
\\
&=-
\Vert P(\bm{\xi}_k \cdot\nabla )\bu \Vert_{L^2}^2.
\end{align*}
This gives a useful cancellation property between the Stratonovich to It\^o corrector and the quadratic variation term that appears when It\^o's formula is used to obtain an energy estimate from \eqref{eq:mainLU}.
}

%\noindent With our preparation in hand, we can now present the system under study. Here, we are interested in a velocity field $\bu:\Omega \times [0,T] \times \mathbb{T}^d\rightarrow \mathbb{R}^d$ that satisfy 
%\begin{equation}
%\begin{aligned}
%\label{eq:mainLU}
%    &\dd \bu  +[ P ((\bu\cdot\nabla) \bu) - \nu  \Delta \bu   ]\dt =
%    \frac{1}{2}\sum_{k\geq1} P((\bm{\xi}_k \cdot\nabla )(\bm{\xi}_k \cdot\nabla )\bu)  \dt
%    \\&\qquad\qquad\qquad\qquad\qquad\qquad
%    + \sum_{k\geq1}P\big[\bg_k(\bu)-((\bm{\xi}_k \cdot\nabla )\bu)  \big] \dd W_t^k,
%    \\
%    &\divx \bu=0,\\
%    &\bu(\bx,0)= \bu_0(\bx).
%\end{aligned}
%\end{equation}
%Here, $\nu>0$ is the viscosity coefficient and $ P = I- \nabla \Delta^{-1}\mathrm{div}$ is the Leray projector with periodic boundary condition. Details on the body forces $(\bg_k)_{k\geq1}$ and the coefficients of the noise advection $(\bm{\xi}_k)_{k\geq1}$ will be given in the next section.

\subsection{Assumptions}
\label{sec:assumptions}
\noindent 
{ Let $\sigma_1\geq 0$, $s>0$ and $r\geq0$ be parameters that will be chosen later. 
For the family of mappings $(\bg_k)_{k\geq1}: [0,T] \times\mathbb{T}^d\times  \mathbb{R}^d  \rightarrow   \mathbb{R}^d $}, the following estimates
\begin{align}
&\sum_{k\geq1}\Vert e^{\sigma_1 A^{1/2s}}  \bg_k(t,\bx, \bv) \Vert_{H^r} 
\lesssim 1+\Vert 
 e^{\sigma_1 A^{1/2s}} \bv  \Vert_{H^r},
&\qquad (Growth) \label{noiseGrowth}
\\&
\sum_{k\geq1} \Vert e^{\sigma_1 A^{1/2s}} [\bg_k(t,\bx, \bv)- \bg_k(t,\bx, \mathbf{w}) ]\Vert_{ H^r } 
\nonumber
\\&\qquad\qquad\qquad\qquad\qquad\qquad\lesssim \Vert 
e^{\sigma_1 A^{1/2s}} [\bv- \mathbf{w}] \Vert_{H^r}
&\qquad (Lipschitz) \label{noiseLipschitz}
\end{align}
holds a.s.  for any $\bv, \mathbf{w} \in  D(e^{\sigma_1  A^{1/2s} } : \mathbb{H}^r(\mathbb{T}^d))$, where the constants in \eqref{noiseGrowth}--\eqref{noiseLipschitz} are independent of $t\in[0,T]$.
Additionally, we assume that
\begin{align}
%&\bg_k(t,\bv(\bx+2\pi \bk))=\bg_k(t,\bv(\bx)) 
%\text{ for all } \bx\in \mathbb{R}^d, \quad \bk\in \mathbb{Z}^d,\quad k\geq 1,
%&\qquad (Periodic)
%\\
&\int_{\mathbb{T}^d}\bg_k(t,\bx,\bv) \dx=0\text{ for each } k\geq 1,
&\qquad (Mean-free)
\label{noiseMeanfree}
\end{align}

\begin{remark}
The assumptions imposed on the noise coefficients $(\bg_k)_{k\geq 1}$ { cover but are not limited to}  the  cases of additive and linear multiplicative noise.
\end{remark}
%\noindent With  regards to the transport noise, we  require that the sum of the coefficients $(\bm{\xi}_k)_{k\geq 1}$ is spatially analytic and each vector is solenoidal. To ensure that this is the case, for simplicity, we assume that these coefficients are time-independent and that for any  $r\geq 0, s>0$ and $\sigma_2\geq0$,
{ 
\noindent With  regards to the transport noise, we  require that the sum of the coefficients $(\bm{\xi}_k)_{k\geq 1}$ is spatially analytic and each vector is solenoidal. For example, $(\bm{\xi}_k)_{k\geq1}$   encodes data from satellites, drifters, and floats in large-scale ocean dynamics \cite{holm}.
%To ensure that this is the case, for simplicity, we assume that these coefficients are constant vectors so that 
More precisely, we assume that for any  $r\geq 0, s>0$ and $\sigma_2\geq0$,
\begin{align}
\label{xikbounded}
\sum_{k\geq1}\Vert e^{\sigma_2 A^{1/2s}}  \bm{\xi}_k \Vert_{H^{r}}\leq K<\infty, \qquad \divx\bm{\xi}_k=0
\end{align}
holds for a constant $K>0$.
%Furthermore, in order to handle the transport noise term, we also make the following strong assumption on the commutativity of the spatial differential operators and the noise advection coefficients. We assume that for each  $k\in \mathbb{N}$,
%{ 
%Furthermore, since $(\bm{\xi}_k)_{k\geq 1}$ are constant vectors, it immediately follow that for each  $k\in \mathbb{N}$
%}
%\begin{align} 
%\label{strongAssumption}
%A^re^{\sigma_2 A^{1/2s}}((\bm{\xi}_k\cdot\nabla)\bv)=
%(\bm{\xi}_k\cdot\nabla)A^re^{\sigma_2 A^{1/2s}}\bv
%\end{align}
%holds for any $\bv\in  D(e^{\sigma_1  A^{1/2s} } : \mathbb{H}^{r+1}(\mathbb{T}^d))$.  

%Finally, since $(\bm{\xi}_k\cdot\nabla)\bv=\mathrm{div}( \bm{\xi}_k\otimes\bv)$ for $\div\bm{\xi}_k=0$, by Helmoltz decomposition, either $\bm{\xi}_k\otimes\bv$ is a
% we note that
Finally, in order to treat the multiplication noise and the transport noise separately, we impose an orthogonality constraint between them. More precisely, we assume that 
\begin{align}
\label{orthoTwoNoise}
\big\langle
A^re^{\sigma_1 A^{1/2s}}\bg_k(\bv)\,,\,
A^re^{\sigma_1 A^{1/2s}}((\bm{\xi}_k\cdot\nabla)\mathbf{w})
\big\rangle_{L^2}=0
\end{align}
holds for any $k\geq1$ and any $\bv,\mathbf{w}   \in  D(e^{\sigma_1  A^{1/2s} } : \mathbb{H}^r(\mathbb{T}^d))$.

\subsection{Concepts of solution}
We now make precise the various notion of a solution that we shall refer to throughout this work.
\begin{definition}[Local strong pathwise solution]
\label{def:locStrongPathSol}
Let $\tau$ be an a.s. strictly positive stopping time and let $\bu :[0,\tau) \rightarrow \mathbb{H}^r(\mathbb{T}^d)$ be a stochastic process for some $r\geq1$. We call the pair $(\bu, \tau)$  a \textit{local strong pathwise solution} of \eqref{eq:mainLU} if { there exists an increasing sequence} of strictly positive stopping times $(\tau_l)_{l\geq1}$ such that:
\begin{itemize}
    \item  $\tau_l <\tau$   and $\lim_{l\rightarrow \infty} \tau_l = \tau$ a.s.;
    \item $\bu$ is  $(\mathcal{F}_t)$-progressively measurable and for each $l\geq1$,
    \begin{align*}
        \bu(\cdot \wedge \tau_l) \in C([0,T];\mathbb{H}^r(\mathbb{T}^d)) \cap L^2((0,T) ;\mathbb{H}^{r+1} (\mathbb{T}^d)) \qquad
         \mathbb{P}\text{-a.s.;}
    \end{align*}
    \item For each $l\geq1$, the equation
    %\begin{equation}
    %\begin{aligned}
    %\big\langle \bu(t \wedge \tau_l), \bm{\phi} \big\rangle_{L^2}
  %&+ \nu\int_0^{t \wedge \tau_l} \big\langle \nabla \bu(\sigma), \nabla\bm{\phi} \big\rangle_{L^2} \ds + \int_0^{t  \wedge \tau_l} \big\langle (\bu(\sigma), \nabla) \bu(\sigma), \bm{\phi} \big\rangle_{L^2} \ds  
  %= 
  %\big\langle \bu_0, \bm{\phi} \big\rangle_{L^2}
  %\\& -
  %\frac{1}{2}
%\sum_{k\geq1}  \int_0^{t \wedge \tau_l}  \big\langle  (\bm{\xi}_k \cdot\nabla )\bu(\sigma), (\bm{\xi}_k \cdot\nabla )\bm{\phi} \big\rangle_{L^2} \ds
%\\& +
%\sum_{k\geq1}  \int_0^{t \wedge \tau_l} \big\langle  \bg_k(\bu(\sigma)), \bm{\phi} \big\rangle_{L^2} \dd W_\sigma^k
%-
%\sum_{k\geq1}  \int_0^{t  \wedge \tau_l}\big\langle  (\bm{\xi}_k \cdot\nabla )\bu(\sigma), \bm{\phi} \big\rangle_{L^2} \dd W_\sigma^k
%\end{aligned}
%\end{equation}
\begin{equation}
    \begin{aligned}
    \bu(t \wedge \tau_l)
  &=\bu_0
  - \int_0^{t  \wedge \tau_l}   P((\bu , \nabla) \bu) 
  \ds  
  + \nu\int_0^{t \wedge \tau_l} \Delta\bu \ds 
  \\&+
  \frac{1}{2}
\sum_{k\geq1}  \int_0^{t \wedge \tau_l}     P((\bm{\xi}_k \cdot\nabla )P(\bm{\xi}_k \cdot\nabla )\bu) \ds
\\& +
\sum_{k\geq1}  \int_0^{t \wedge \tau_l}  P \bg_k(\bu ) \dd W_\sigma^k
-
\sum_{k\geq1}  \int_0^{t  \wedge \tau_l}  P( (\bm{\xi}_k \cdot\nabla )\bu) \dd W_\sigma^k
\end{aligned}
\end{equation}
    holds a.s. for all $t\in[0,T]$.
    %and  for all $\bm{\phi} \in C^\infty_{\divx}(\mathbb{T}^d)$.
\end{itemize}
\end{definition}

\begin{definition}[Maximal strong pathwise solution]
\label{def:MaxStrongPathSol}
A local strong pathwise solution $(\bu,\tau)$ of \eqref{eq:mainLU} is called a \textit{maximal strong pathwise solution} of  \eqref{eq:mainLU} if for any other local strong pathwise solution $(\bu',\tau')$, we have
\begin{align} 
\label{maximal}
\tau' \leq \tau \quad\text{ a.s. and } \quad \bu=\bu' \quad \text{ a.e. on } \quad\Omega \times [0,\tau')\times \mathbb{T}^d.
\end{align}
\end{definition}
\begin{remark}
%Note that by definition, maximal strong pathwise solutions are automatically unique in the class of local strong pathwise solutions.
%
{ 
Note that by definition, maximal strong pathwise solutions are \textit{essentially unique} in the class of local strong pathwise solutions in the following sense. For any two maximal strong pathwise solutions $(\mathbf{u}_1,\tau_1)$ and $(\mathbf{u}_2,\tau_2)$, it follows from \eqref{maximal} that 
\begin{align*}
\tau_1=\tau_2=\tau\quad\text{ a.s. and } \quad\mathbf{u}_1=\mathbf{u}_2  \quad \text{ a.e. on } \quad\Omega\times[0,\tau)\times\mathbb{T}^d.
\end{align*}
But one could say even more. Using continuity of the trajectories of $\mathbf{u}_1$ and $\mathbf{u}_2$ - as $\mathbb{H}^r(\mathbb{T}^d)$-
valued random variables - one could replace the condition $\mathbf{u}_1=\mathbf{u}_2$ a.e. on
$\Omega\times[0,\tau)\times\mathbb{T}^d$ with the stronger statement that there exists a full measure set $\tilde{\Omega}\subseteq \Omega$
such that, for every $\omega\in \tilde{\Omega}$ it holds $\mathbf{u}_1(\omega,t)=\mathbf{u}_2(\omega,t)$ (as an identity in $\mathbb{H}^r(\mathbb{T}^d)$) for every $t<\tau(\omega)$. In particular, the two processes $\mathbf{u}_1$ and $\mathbf{u}_2$ are
indistinguishable.
}
\end{remark}
\noindent In the following, we state a result from \cite{glatt2009strong} on the existence of an essentially unique local (global when $d=2$) strong pathwise solution of \eqref{eq:mainLU} without the transport noise term. { An extension to the full system \eqref{eq:mainLU} applies} and coincides with the construction that is to be performed in Section \ref{sec:GevReg} when the radius of analyticity is $\varphi=0$.
\begin{theorem}[An essentially unique local strong pathwise solution] 
\label{thm:globStrong} 
 Let $(\bg_k)_{k\geq1}$ satisfy \eqref{noiseGrowth}--\eqref{noiseMeanfree} with  $r=1$, $s>0$ and $\sigma_1=0$ and let $(\bm{\xi}_k)_{k\geq1}$ satisfy \eqref{xikbounded} with $r>4$, $s>0$ and $\sigma_2=0$.
For $\bu_0\in L^p_{\mathcal{F}_0}\big(\Omega; \mathbb{H}^1 (\mathbb{T}^d)\big)$ with $p\in(1,\infty)$, there exists an { essentially unique local strong pathwise
solution $(\bu,\tau)$ of \eqref{eq:mainLU}
such that
\begin{align}
\label{h1Energy}
\mathbb{E}\bigg[ \sup_{t \in [0, T\wedge \tau]} \Vert   \bu \Vert_{H^1}^p
+
\int_0^{T\wedge \tau}\Vert  \bu \Vert_{H^2}^2 \Vert   \bu \Vert_{H^1}^{p-2} \dt \bigg] \leq c
\end{align}
}
holds with a constant $c=c(\bu_0, K)$ where $K$ is the upper bound appearing in \eqref{xikbounded}.
\end{theorem}
\noindent We can now make precise, what we mean by a solution of \eqref{eq:mainLU} being of Gevrey class $s$.
\begin{definition}[Gevrey class $s$ solution]
\label{def:GevreySol}
Let  $(\bg_k)_{k\geq1}$ and  $(\bm{\xi}_k)_{k\geq1}$ satisfy the assumptions in Section \ref{sec:assumptions} for some $r\geq0$, $\sigma_1,\sigma_2>0$ and $s>0$. Let $\tau$ be an a.s. strictly positive stopping time and let $\bu :[0,\tau) \rightarrow \mathbb{H}^r(\mathbb{T}^d)$ be a stochastic process with $r\geq1$. We call the pair $(\bu, \tau)$  a \textit{Gevrey class $s$ solution} of \eqref{eq:mainLU} if:
\begin{enumerate}
\item $(\bu,\tau)$ is  a { local} strong pathwise solution of \eqref{eq:mainLU}; 
    \item { $\bu(\cdot \wedge \tau )\in  C((0,\tau ); D(e^{\varphi(\cdot)  A^{1/2s}} : \mathbb{H}^r(\mathbb{T}^d)))$ a.s.}
\end{enumerate}
\end{definition}

\begin{remark}
Henceforth, we will set $r=s=1$ and $\varphi(t)=t$  and concentrate on the Gevrey class $1$ solution. Please see the preliminary section of \cite{levermore1997analyticity} for further discussions on the relationship between this Gevrey class and the space of analytic functions.
\end{remark}
\begin{remark}
\label{rem:weakContinuity}
Here, $C((0,\tau ); D(e^{t A^{1/2s}} : \mathbb{H}^r(\mathbb{T}^d)))$ denotes the set of functions $\bu : (0, \tau ) \rightarrow D(e^{t  A^{1/2s}} : \mathbb{H}^r(\mathbb{T}^d))$ being continuous with respect
to the norm topology, i.e.,
{ 
\begin{align*}
   e^{t_k  A^{1/2s}} \bu(t_k) \rightarrow e^{t  A^{1/2s}}\bu(t) \quad\text{in}\quad   \mathbb{H}^r(\mathbb{T}^d)
\end{align*}
}
for any sequence $(t_k)_{k\in \mathbb{N}} \subset (0,\tau )$ with $t_k\rightarrow t$.
\end{remark}
\subsection{Main results}
Our first  main theorem is the following.
\begin{theorem}[An essentially unique  Gevrey class $1$ solution]
\label{thm:main}Let $(\bg_k)_{k\geq1}$ satisfy \eqref{noiseGrowth}--\eqref{noiseMeanfree} with  $r=s=1$  and $\sigma_1=t $ and let $(\bm{\xi}_k)_{k\geq1}$ satisfy \eqref{xikbounded} with $r>4$, $s=1$ and $\sigma_2=T$. Now suppose that there exists a deterministic $K_0>0$ such that 
\begin{align}
\label{initialCondBounded}
\Vert \bu_0 \Vert_{H^1}^2\leq K_0 \qquad \text{a.s.}
\end{align}
Then there exists an essentially unique   Gevrey class $1$ (real-analytic) solution $(\bu,\tau)$ of \eqref{eq:mainLU}.
\end{theorem}
\begin{remark} 
\label{rem:thm:main} 
Note that due to Theorem \ref{thm:globStrong}, the proof of Theorem \ref{thm:main} is done once we show the second item  in Definition \ref{def:GevreySol}.
\end{remark}
\begin{remark}
It is possible to generalize the assumption \eqref{initialCondBounded}   to $\bu_0\in L^p_{\mathcal{F}_0}\big(\Omega; \mathbb{H}^1 (\mathbb{T}^d)\big)$ with $p\in(1,\infty)$  by employing a truncating argument in the spirit of \cite{glatt2009strong}. That is, for $\bu_0\in L^p_{\mathcal{F}_0}\big(\Omega; \mathbb{H}^1 (\mathbb{T}^d)\big)$, we define
\begin{align*}
\bu_0^k:= \bu_0 \bm{1}_{\{k \leq \Vert   \bu_0 \Vert_{H^1} <k+1\}}
\end{align*}
for $k\geq0$ so that the  boundedness assumption \eqref{initialCondBounded} holds for $\bu_0^k$. Subsequently, as we shall soon see, this will generate a Gevrey class $1$ solution $(\bu^k,\tau^k)$   of \eqref{eq:mainLU}. We then define the pair
\begin{align*}
\bu:= \sum_{k\geq0}\bu^k \bm{1}_{\{k \leq \Vert   \bu_0 \Vert_{H^1} <k+1\}},
\\
\tau:= \sum_{k\geq0}\tau^k \bm{1}_{\{k \leq \Vert   \bu_0 \Vert_{H^1} <k+1\}},
\end{align*}
and find that $(\bu, \tau)$ is indeed a Gevrey class $1$ solution of \eqref{eq:mainLU} with the general initial condition $\bu_0\in L^p_{\mathcal{F}_0}\big(\Omega; \mathbb{H}^1 (\mathbb{T}^d)\big)$. However, for clarity of exposition, we impose the bounded initial condition \eqref{initialCondBounded} throughout the paper.
\end{remark}
\noindent By using the fact that the Gevrey class $1$ functions are contained in the space of smooth functions, an immediate corollary of Theorem \ref{thm:main} is the following smoothing result. See \cite[Lemma 3]{levermore1997analyticity} for more details.
\begin{corollary}[Instantaneous smoothing]
\label{thm:instSmooth}
Let assumptions of Theorem \ref{thm:main} hold.
Any  Gevrey class $1$  solution  $(\bu,\tau)$ of \eqref{eq:mainLU} satisfies $\bu(\cdot\wedge \tau)\in C((0, \tau );C^\infty_{\divx}( \mathbb{T}^d))$ a.s. 
\end{corollary}
\noindent
Finally, by using the spatial discretization in the construction of the above Gevrey class regularity, we
give a quantitative error estimate for the difference between the original continuum solution and the solution to
the equation solved by the truncated finite-dimensional approximation. This estimate is uniform over the whole time interval $[0, T]$, uniform with respect to the initial enstrophy, and uniform with respect to the noise coefficients. In particular, uniformly with respect to the aforementioned parameters, we can infer that the Fourier coefficient $\hat{\bu}_\bk(t)$ decay exponentially with respect to its Fourier modes or wavenumber $\vert \bk\vert$. The exact statement of the result is given in Theorem \ref{thm:decay} below.
To state this theorem, we let $\bu^N$ be a unique solution to a finite-dimensional Galerkin approximation of \eqref{eq:mainLU}. Further details on this approximation will be given in Section \ref{sec:GevReg} before we give the proof of Theorem \ref{thm:decay}. 
\begin{theorem}[Rate of decay]
\label{thm:decay} 
Let the assumptions in Theorem \ref{thm:main} hold so that for $\bu_0 $ satisfying \eqref{initialCondBounded}, $(\bu,\tau)$ is the corresponding essentially unique Gevrey class $1$  solution  of \eqref{eq:mainLU} obtained as the limit of the Galerkin approximation $\bu^N$.
Assume that $\tau=\infty$ $\mathbb{P}$-a.s. 
and for each $R>0$, let 
\begin{align*}
\tau_R:=\inf \bigg\{ t\in(0,T)\,:\, \int_0^t \big(\Vert\bu^N(\sigma)\Vert_{H^2}^2 + \Vert\bu(\sigma)\Vert_{H^2}^2 \big)\ds\geq R\bigg\}
\end{align*}
be a stopping time.
% with the property that there exists a deterministic $\delta_*>0$ such that
%\begin{align}
%\label{tauRzero}
%\mathbb{P}[\tau_R\leq \delta_*]=0.
%\end{align}
Then there exist a constant $c=c(\nu,K_0,K)
e^{c(\nu,K)(T+
R
 )}$ such that for any  $t\in(0,T]$, we have that
 {  
\begin{align*}
\mathbb{E} 
\Vert( \bu - \bu^N)(t\wedge \tau_R)\Vert_{H^1}^2 \leq c\,
%\mathbb{E} e^{-2\mathtt{t}N}
e^{-2\delta N}
\end{align*}
%where $\mathtt{t}=\delta\wedge \tau_R$
for some deterministic time $\delta>0$.}
\end{theorem}

\section{Preparatory results}
\label{sec:preparation}
\noindent Before  proving our main results, we collect in this section, various essential results that will be needed in the sequel. To begin with, we give an estimate for the convective term in the Navier-Stokes equation that is useful in showing the required Gevrey class regularity. The proof can be found in \cite[Lemma 2.1]{foias1989gevrey}.
\begin{remark}
Henceforth, we drop the $L^2$ in the inner product $\langle\cdot,\cdot\rangle_{L^2}$.
\end{remark}
\begin{lemma}
\label{lem:main}
Let  $\bu,\bv,\bw \in  D(e^{\varphi  A^{1/2} } : \mathbb{H}^2(\mathbb{T}^d))$ with $\varphi\geq0$.
We have\\ $P((\bu\cdot\nabla) \bv )\in  D(e^{\varphi  A^{1/2} } : \mathbb{L}^2(\mathbb{T}^d))$ and
 the inequality 
\begin{align*}
\big\vert\big\langle  A^{1/2}  e^{\varphi  A^{1/2}  }  P((\bu\cdot\nabla) \bv) \,,\, A^{1/2}  e^{\varphi  A^{1/2}  } \bw \big\rangle\big\vert
&\lesssim
\Vert  A^{1/2}  e^{\varphi  A^{1/2}  }\bu \Vert_{L^2}^\frac{1}{2}
\Vert  A  e^{\varphi  A^{1/2}  }\bu \Vert_{L^2}^\frac{1}{2}
\\
&\times
\Vert  A^{1/2}  e^{\varphi  A^{1/2}  }\bv \Vert_{L^2}
\Vert  A  e^{\varphi  A^{1/2}  }\bw \Vert_{L^2}.
\end{align*}
holds.
\end{lemma}
\noindent The next lemma, whose proof can be found in the appendix, is a further useful estimate for the deterministic nonlinear term in the Navier--Stokes equation.
\begin{lemma}
\label{lem:h1convectiveX}
For $\bu,\bv\in D(e^{\varphi  A^{1/2}} : \mathbb{H}^1(\mathbb{T}^d))$ with $\varphi\geq0$, it holds that
\begin{align*}
\Vert A^{1/2} e^{\varphi A^{1/2}} (\bu\cdot\bv)\Vert_{L^2}
&\lesssim
\Vert  e^{\varphi A^{1/2}} \bu \Vert_{L^2}
\Vert  e^{\varphi A^{1/2}} \bv \Vert_{H^1}
+
\Vert  e^{\varphi A^{1/2}}\bu \Vert_{H^1}
\Vert  e^{\varphi A^{1/2}} \bv \Vert_{L^2}.
\end{align*}
\end{lemma}
\noindent Our next goal is to give a cancellation property
%, { relying on the expression \eqref{strongAssumption},} 
for the sum of the quadratic variation  term and the Stratonovich-to-It\^o's correction term when we apply It\^o's formula to the mapping $t \mapsto \frac{1}{2} \Vert  A^{1/2}  e^{\varphi A^{1/2}  } \bu(t) \Vert^2_{L^2} $. Its proof can also be found in the appendix, Section \ref{sec:appendix}.
\begin{lemma} 
\label{cor:xiAppendix}
Let $r\geq 0, s>0$ and $\varphi\geq 0$. For any
 $ \bm{\xi}_k\in D(e^{\varphi  A^{1/2s}} : \mathbb{H}^{r+2}(\mathbb{T}^d))$
and $\bu\in D(e^{\varphi  A^{1/2s}} : \mathbb{H}^{r+2}(\mathbb{T}^d))$, we have that
\begin{align*}
 \big\langle
 A^r   e^{\varphi  A^{1/2s}  }
     ((\bm{\xi}_k \cdot \nabla) (\bm{\xi}_k \cdot \nabla)\bu )
\, ,\, &  A^r   e^{\varphi  A^{1/2s}  } \bu
\big\rangle 
+
\big\Vert
 A^r  e^{\varphi  A^{1/2s}  }
    ((\bm{\xi}_k \cdot \nabla)\bu)
\big\Vert_{L^2}^2 =0
\end{align*}
for any $k\geq1$
\end{lemma} { 
\noindent Our final preparatory result is an adaptation of \cite[Lemma 5.1]{glatt2009strong} to Gevrey spaces. Compare with \cite[Lemma 7.1]{glatt2014local} and \cite[Lemma 2.1]{breit2019stochastic} and see the appendix below for the proof.
}
In the following, we fix $r\geq0$, $s>0$, $\nu\geq0$ and let  $\varphi\geq0$ be an arbitrary time-dependent function of bounded variation. 
Now define $\mathcal{E}(T)$ as
\begin{align*}
\mathcal{E}(T)
:=
C\big([0,T];D(e^{\varphi(t)  A^{1/2s}} : \mathbb{H}^r(\mathbb{T}^d))\big) \cap L^2\big((0,T);D(e^{\varphi(t)  A^{1/2s}} : \mathbb{H}^{r+1}(\mathbb{T}^d))\big)
\end{align*}
with the norm
\begin{align*}
\Vert \bu^N  \Vert_{\mathcal{E}(\tau)}^2
&:=
\sup_{\sigma\in[0,{  \tau} ]}\Vert A^{r/2}  e^{\varphi(\sigma) A^{1/2s}} \bu^N(\sigma) \Vert^2_{L^2} 
\\&+
\nu
\int_0^{ \tau}
\Vert  A^{(r+1)/2} e^{\varphi(\sigma)A^{1/2s}}
 \bu^N
\Vert_{L^2}^2 \ds.
\end{align*}

\begin{lemma}
\label{lem:comparison}
Let $s>0$ and $r,\varphi\geq 0$. Let $(\Omega, \mathcal{F}, (\mathcal{F}_t)_{t\geq 0}, \mathbb{P})$ be a stochastic basis with the complete right-continuous filtration $(\mathcal{F}_t)_{t\geq 0}$. Let $\bu^N$ be an $(\mathcal{F}_t)$-adapted continuous stochastic process valued in
$D(e^{\varphi(t)  A^{1/2s}} : \mathbb{H}^{r+1}(\mathbb{T}^d))$.
For a deterministic $M>1$ and $T>0$,
define the stopping times $\mathcal{T}^M_N$ as
\begin{align*}
\mathcal{T}^M_N:=\inf \bigg\{  \tau\in [0,T]\,:\,
\Vert \bu^N \Vert_{\mathcal{E}(\tau)}
>
\Vert   e^{\varphi(0)A^{1/2s} }\bu^N_0 \Vert_{H^r} 
+ M
\bigg\} \wedge T
\end{align*}
and set $\mathcal{T}^M_{Nn}:= \mathcal{T}^M_N \wedge \mathcal{T}^M_n$. If
 \begin{equation}
\begin{aligned}
\label{probZero}
\lim_{t\rightarrow0}\sup_{N\in \mathbb{N}}
\mathbb{P} \bigg(
\Vert \bu^N \Vert_{\mathcal{E}(t \wedge \mathcal{T}^M_N)}>
\Vert  e^{\varphi(0)A^{1/2s} }\bu^N_0 \Vert_{H^r} 
+ (M-1)\bigg)
=0
 \end{aligned}
\end{equation}
and
\begin{equation}
\begin{aligned}
\label{expoZero}
\lim_{n\rightarrow\infty}\sup_{N\geq n}
\mathbb{E} \Vert  \bu^N -\bu^n  \Vert_{\mathcal{E}(\mathcal{T}^M_{Nn})} 
=0
 \end{aligned}
\end{equation}
hold, then 
\begin{itemize}
\item there exists a stopping time $\tau$ with
\begin{align}
\label{strictlyPositiveStopTime}
\mathbb{P}(0<\tau \leq T)=1
\end{align}
\item there exists a process $\bu(\cdot)=\bu(\cdot\wedge \tau) \in \mathcal{E}(\tau)$ such that up to taking a subsequence (not relabelled),
\begin{align}
\label{cauchyZero}
\Vert \bu^N - \bu  \Vert_{\mathcal{E}(\tau)} \rightarrow 0 \qquad a.s.;
\end{align}
\item also,
\begin{align}
\label{cauchyZeroM}
\Vert  \bu \Vert_{\mathcal{E}(\tau)} \leq 
M +
\sup_{N}\Vert   e^{\varphi(0)A^{1/2s} }\bu^N_0 \Vert_{H^r} 
 \qquad a.s.
\end{align}
\end{itemize}
\end{lemma}

\section{Gevrey regularity: the construction}
\label{sec:GevReg}
\subsection{Estimates}
With an initial condition $\bu_0\in L^\infty_{\mathcal{F}_0}\big(\Omega; \mathbb{H}^1(\mathbb{T}^d)\big)$ and a dataset $(\bg_k, \bm{\xi}_k)_{k\geq1}$ satisfying the assumptions in Section \ref{sec:assumptions}, our goal now is to construct a solution $\bu$ of \eqref{eq:mainLU} that lives in the Bochner space $C([0,T); D(e^{t  A^{1/2}} : \mathbb{H}^1(\mathbb{T}^d)))$ a.s.
We will achieve this result by using a Galerkin approximation. In particular, we let $P^N:L^2(\mathbb{T}^d) \rightarrow H_N$ be the $L^2$-orthogonal projection onto 
 $H_N=\text{span}\{e^{i\bk\cdot \bx} \, \text{ with }\,   \vert\bk \vert \leq N\}$, and consider the finite-dimensional stochastic  differential equation (SDE)
\begin{equation}
\begin{aligned}
\label{eq:galerkin}
    &\dd \bu^N  +[  P^N P ((\bu^N\cdot\nabla) \bu^N)  -\nu  \Delta \bu^N   ]\dt 
    \\&\qquad\qquad\qquad\qquad
    =
    \frac{1}{2}\sum_{k\geq1} P^NP((\bm{\xi}_k \cdot\nabla )P(\bm{\xi}_k \cdot\nabla )\bu^N)  \dt
    \\&\qquad\qquad\qquad\qquad
    +
     \sum_{k\geq1}P^NP\big[\bg_k(\bu^N)- ((\bm{\xi}_k \cdot\nabla )\bu^N) \big]  \dd W_t^k,
    \\
    &\bu^N_0(\bx)= P^N\bu_0(\bx).
\end{aligned}
\end{equation}
Since \eqref{eq:galerkin} is an $N$-dimensional system  of SDEs, for $\bu_0^N\in L^\infty_{\mathcal{F}_0}(\Omega;H_N )$ we can construct a solution
%global-in-time adapted process 
$\bu^N$ in $L^\infty(0,T;H_N)$ 
that solves \eqref{eq:galerkin} a.e in space-time. See for example, \cite{capinski1994stochastic, flandoli1995martingale}. Furthermore, we can show the following.

\begin{lemma}
\label{lem:criteria}
Suppose that there exists a deterministic $K_0>0$ such that 
\begin{align}
\Vert \bu^N_0 \Vert_{H^1}^2\leq K_0 \qquad \text{a.s.}
\end{align}
For $\bu^N_0\in H_N$, let  $\bu^N\in C([0,T];H_N)$ be the corresponding Galerkin solution of \eqref{eq:galerkin}. Now  for any deterministic constant $M>1$, define the stopping times $\mathcal{T}^M_N$ as
\begin{align*}
\mathcal{T}^M_N:=\inf \bigg\{  \tau\in [0,T]\,:\,
\sup_{\sigma\in[0,{  \tau} ]}&\Vert   e^{\sigma  A^{1/2} } \bu^N(\sigma) \Vert^2_{H^1} 
+
\nu  \int_0^{ \tau}
\Vert  e^{\sigma A^{1/2}}
 \bu^N
\Vert_{H^2}^2 \ds
\\&>
\Vert     
%e^{\varphi(0)A^{1/2} }
\bu^N_0 \Vert^2_{H^1} 
+ M
\bigg\} \wedge T
\end{align*}
and set $\mathcal{T}^M_{Nn}:= \mathcal{T}^M_N \cap \mathcal{T}^M_n$. Then for any $\tilde{M}\in(0,M)$,
\begin{align}
\nonumber
\lim_{t\rightarrow0}\sup_{N\in \mathbb{N}}
\mathbb{P} \bigg(\sup_{\sigma\in[0,{t \wedge \mathcal{T}^M_N} ]}\Vert    e^{\sigma  A^{1/2} } \bu^N(\sigma) \Vert^2_{H^1} 
&+
\nu  \int_0^{t \wedge \mathcal{T}^M_N}
\Vert    e^{\sigma A^{1/2}}
 \bu^N
\Vert_{H^2}^2 \ds
\\
&>
\Vert   
%e^{\varphi(0) A^{1/2}}
\bu^N_0 \Vert^2_{H^1} 
+ \tilde{M}\bigg)
=0
\label{probZeroX}
 \end{align}
and
\begin{equation}
\begin{aligned} 
\label{expecZeroX}
\lim_{n\rightarrow\infty}\sup_{N\geq n}
\mathbb{E} \bigg(&\sup_{\sigma\in[0,{ \mathcal{T}^M_{Nn}} ]}\Vert   e^{\sigma A^{1/2} } [\bu^N(\sigma)-\bu^n(\sigma)] \Vert^2_{H^1} 
\\&+
\nu  \int_0^{\mathcal{T}^M_{Nn}}
\Vert    e^{\sigma  A^{1/2} }
 [\bu^N-\bu^n]
\Vert_{H^2}^2 \ds
\bigg)
=0.
 \end{aligned}
\end{equation}
\end{lemma}
\begin{proof}
By applying It\^o's formula   to the mapping $t\mapsto \frac{1}{2}  \Vert    e^{\sigma  A^{1/2}  } \bu^N(t) \Vert^2_{H^1} $, it follows from \eqref{eq:galerkin} that
\begin{align}
\frac{1}{2}& \Vert    e^{\sigma  A^{1/2}  } \bu^N(t) \Vert^2_{H^1} 
+
\nu \int_0^t \Vert
 e^{\sigma  A^{1/2}  }
     \bu^N
\Vert_{H^2} \ds
= 
\frac{1}{2} \Vert    
 %e^{\varphi(0) A^{1/2}  }
 \bu^N_0 \Vert^2_{H^1}
  \nonumber
\\
&-
\int_0^t \big\langle
 A^{1/2}   e^{\sigma  A^{1/2}  }
    P^NP (\bu^N\cdot\nabla) \bu^N
,\,   A^{1/2}   e^{\sigma  A^{1/2}  } \bu^N 
\big\rangle \ds
\nonumber
\\
&+
\int_0^t \big\langle
  A   e^{\sigma  A^{1/2}  } \bu^N 
,\,   A^{1/2}   e^{\sigma  A^{1/2}  } \bu^N 
\big\rangle \ds
\nonumber
\\
&+
\frac{1}{2}\int_0^t
\sum_{k\geq 1}
\big\Vert
   e^{\sigma  A^{1/2}  }
   P^NP[\bg_k(\bu^N)
   -
   ((\bm{\xi}_k \cdot \nabla)\bu^N)]
\big\Vert_{H^1}^2 \ds
\nonumber
\\
&+
\frac{1}{2}\int_0^t\sum_{k\geq 1}
 \big\langle
 A^{1/2}   e^{\sigma  A^{1/2}  }
   P^N P[(\bm{\xi}_k \cdot \nabla)P(\bm{\xi}_k \cdot \nabla)\bu^N]
,\,   A^{1/2}   e^{\sigma  A^{1/2}  } \bu^N 
\big\rangle \ds
\nonumber
\\
&+\int_0^t\sum_{k\geq 1}
 \big\langle
 A^{1/2}   e^{\sigma  A^{1/2}  }
   P^NP\bg_k(\bu^N)
,\,   A^{1/2}   e^{\sigma  A^{1/2}  } \bu^N 
\big\rangle \dd W^k_\sigma
\nonumber
\\
&-\int_0^t\sum_{k\geq 1}
 \big\langle
 A^{1/2}   e^{\sigma  A^{1/2}  }
   P^N P((\bm{\xi}_k \cdot \nabla)\bu^N)
,\,   A^{1/2}   e^{\sigma  A^{1/2}  } \bu^N 
\big\rangle \dd W^k_\sigma.
\label{apriori1}
 \end{align}
We now estimate the various deterministic integrals to the right of the equation above. First of all,  note that by using  Lemma \ref{lem:main}, we have 
\begin{equation}
\begin{aligned}  
\label{convectEst}
&\int_0^{t \wedge \mathcal{T}^M_N}\big\vert\big\langle  A^{1/2}  e^{\sigma  A^{1/2}  } P^NP ((\bu^N
\cdot \nabla)\bu^N ),  A^{1/2}  e^{\sigma  A^{1/2}  } \bu^N \big\rangle \big\vert\ds
\\&\leq
\frac{\nu}{4}
\int_0^{t \wedge \mathcal{T}^M_N}
\Vert     e^{\sigma  A^{1/2}  }\bu^N \Vert_{H^2}^2
\ds
+c(\nu)\int_0^{t \wedge \mathcal{T}^M_N}
\Vert  e^{\sigma  A^{1/2}  }\bu^N  \Vert_{H^1}^6\ds.
\end{aligned}
\end{equation} 
On the other hand, we have
 %by moving $\varphi'(\sigma)$ around 
 that
\begin{equation}
\begin{aligned}
\bigg\vert&\int_0^{t \wedge \mathcal{T}^M_N} \big\langle
 A   e^{\sigma  A^{1/2}  } \bu^N 
,\,   A^{1/2}   e^{\sigma  A^{1/2}  } \bu^N 
\big\rangle \ds \bigg\vert
\\&\leq 
\frac{\nu}{4}  \int_0^{t \wedge \mathcal{T}^M_N}
\Vert     e^{\sigma  A^{1/2}  }
 \bu^N
\Vert_{H^2}^2 \ds
+ c(\nu)
\int_0^{t \wedge \mathcal{T}^M_N}
\Vert     e^{\sigma  A^{1/2}  }
 \bu^N
\Vert_{H^1}^2 \ds.
\end{aligned}
\end{equation}
Now use the Polarization Identity and \eqref{orthoTwoNoise} to write
\begin{equation}
\begin{aligned}
\label{splitTwo}
\big\Vert
   e^{\sigma  A^{1/2}  }
   P^NP[\bg_k(\bu^N)
   &-
   ((\bm{\xi}_k \cdot \nabla)\bu^N)]
\big\Vert_{H^1}^2
=
\big\Vert
   e^{\sigma  A^{1/2}  }
   P^NP \bg_k(\bu^N)
\big\Vert_{H^1}^2
\\&+
\big\Vert
   e^{\sigma  A^{1/2}  }
   P^N
   P((\bm{\xi}_k \cdot \nabla)\bu^N)
\big\Vert_{H^1}^2.
\end{aligned}
\end{equation}
By using the  assumption on $(\bg_k)_{k\geq1}$ in Section \ref{sec:assumptions} and the continuity of $P^NP$, we can estimate the first term in \eqref{splitTwo} as follows
\begin{equation}
\begin{aligned}
\frac{1}{2}
\int_0^{t \wedge \mathcal{T}^M_N}&
\sum_{k\geq 1}
\big\Vert
   e^{\sigma  A^{1/2}  }
  P^NP \bg_k(\bu^N)
\big\Vert_{H^1}^2 \ds
\\&
\lesssim
\int_0^{t \wedge \mathcal{T}^M_N}
\big(1+
\Vert
       e^{\sigma  A^{1/2}  }
 \bu^N
\Vert_{H^1}^2 \big) \ds.
\end{aligned}
\end{equation}
Next, we combine the second term in \eqref{splitTwo} with the fifth right-hand term in \eqref{apriori1} and estimate both by using  
 Lemma \ref{cor:xiAppendix}. This yields
\begin{equation}
\begin{aligned}
\bigg\vert&\int_0^{t \wedge \mathcal{T}^M_N} \sum_{k\geq 1}
 \big\langle
 A^{1/2}   e^{\sigma  A^{1/2}  }
   P^N P[(\bm{\xi}_k \cdot \nabla) P(\bm{\xi}_k \cdot \nabla)\bu^N]
,\,   A^{1/2}   e^{\sigma  A^{1/2}  } \bu^N 
\big\rangle \ds
\\&+
\sum_{k\geq 1}
\big\Vert
    e^{\sigma  A^{1/2}  }
   P^NP((\bm{\xi}_k \cdot \nabla)\bu^N)
\big\Vert_{H^1}^2 \ds
 \bigg\vert
\lesssim  \int_0^{t \wedge \mathcal{T}^M_N}
\Vert     e^{\sigma  A^{1/2}  }
 \bu^N
\Vert_{H^1}^2 \ds.
\end{aligned}
\end{equation}
By collecting the various estimates above, we obtain
\begin{equation}
\begin{aligned}
\label{estAllapriori}
&\sup_{\sigma\in[0,{t \wedge \mathcal{T}^M_N} ]}\Vert     e^{\sigma  A^{1/2}  } \bu^N(\sigma) \Vert^2_{H^1} 
+
\nu  \int_0^{t \wedge \mathcal{T}^M_N}
\Vert     e^{\sigma  A^{1/2}  }
 \bu^N
\Vert_{H^2}^2 \ds
\leq
\Vert    
 %e^{\varphi(0) A^{1/2}  }
 \bu^N_0 \Vert^2_{H^1} 
\\
&+c(\nu)\int_0^{t \wedge \mathcal{T}^M_N}
\big(1+
\Vert
      e^{\sigma  A^{1/2}  }
 \bu^N
\Vert_{H^1}^6
\big) \ds
\\
&+c(\nu)\sup_{\sigma\in[0,{t \wedge \mathcal{T}^M_N} ]}\bigg\vert\int_0^\sigma\sum_{k\geq 1}
 \big\langle
 A^{1/2}   e^{ s  A^{1/2}  }
  P^NP \bg_k(\bu^N)
,\,   A^{1/2}   e^{ s  A^{1/2}  } \bu^N 
\big\rangle \dd W^k_s
\bigg\vert
\\
&+c(\nu)\sup_{\sigma\in[0,{t \wedge \mathcal{T}^M_N} ]}\bigg\vert\int_0^\sigma\sum_{k\geq 1}
 \big\langle
 A^{1/2}   e^{ s  A^{1/2}  }
  P^N P((\bm{\xi}_k\cdot \nabla)\bu^N)
,\,   A^{1/2}   e^{ s  A^{1/2}  } \bu^N 
\big\rangle \dd W^k_s
\bigg\vert.
 \end{aligned}
\end{equation}
It therefore follow that for every $\tilde{M}\in(0,M)$,
\begin{equation}
\begin{aligned}
\mathbb{P} &\bigg(\sup_{\sigma\in[0,{t \wedge \mathcal{T}^M_N} ]}\Vert   e^{\sigma  A^{1/2}  } \bu^N(\sigma) \Vert^2_{H^1} 
+
\nu  \int_0^{t \wedge \mathcal{T}^M_N}
\Vert     e^{\sigma  A^{1/2}  }
 \bu^N
\Vert_{H^2}^2 \ds
\\&\qquad\qquad\qquad\qquad
>
\Vert     
%e^{\varphi(0) A^{1/2}  }
\bu^N_0 \Vert^2_{H^2} 
+ \tilde{M}\bigg)
\\
&\leq \mathbb{P}\bigg(
c(\nu)\int_0^{t \wedge \mathcal{T}^M_N}
\big(1+
\Vert
     e^{\sigma  A^{1/2}  }
 \bu^N
\Vert_{H^1}^6
\big) \ds
> \frac{\tilde{M}}{3} \bigg)
\\
&+
\mathbb{P}\bigg(
c(\nu)\sup_{\sigma\in[0,{t \wedge \mathcal{T}^M_N} ]}\bigg\vert\int_0^\sigma\sum_{k\geq 1}
 \big\langle
 A^{1/2}   e^{ s  A^{1/2}  }
  P^N P\bg_k(\bu^N)
,\,
\\&\qquad\qquad\qquad\qquad
   A^{1/2}   e^{ s  A^{1/2}  } \bu^N 
\big\rangle \dd W^k_s
\bigg\vert
> \frac{\tilde{M}}{3} \bigg)
\\
&+
\mathbb{P}\bigg(
c(\nu)\sup_{\sigma\in[0,{t \wedge \mathcal{T}^M_N} ]}\bigg\vert\int_0^\sigma\sum_{k\geq 1}
 \big\langle
 A^{1/2}   e^{ s  A^{1/2}  }
  P^NP( (\bm{\xi}_k\cdot \nabla)\bu^N)
,\,   
\\&\qquad\qquad\qquad\qquad
A^{1/2}   e^{ s  A^{1/2}  } \bu^N 
\big\rangle \dd W^k_s
\bigg\vert
> \frac{\tilde{M}}{3} \bigg).
 \end{aligned}
\end{equation}
To estimate the first term on the right, we use Chebyshev's inequality just as in \cite{glatt2009strong} which yields
\begin{equation}
\begin{aligned}
\mathbb{P}\bigg(
c(\nu)\int_0^{t \wedge \mathcal{T}^M_N}&
\big(1+
%\Vert
%     e^{\sigma  A^{1/2}  }
% \bu^N
%\Vert_{H^1}^2 +
%\Vert
%\varphi'(\sigma)
%      e^{\sigma  A^{1/2}  }
% \bu^N
%\Vert_{H^1}^2
%+
\Vert
      e^{\sigma  A^{1/2}  }
 \bu^N
\Vert_{H^1}^6
\big) \ds
> \frac{\tilde{M}}{2} \bigg)
\\&\leq
c_{\nu,\tilde{M}}
\mathbb{E}
\int_0^{t \wedge \mathcal{T}^M_N}
\big(1+
%\Vert
%     e^{\sigma  A^{1/2}  }
% \bu^N
%\Vert_{H^1}^2 +
%\Vert
%\varphi'(\sigma)
%     e^{\sigma  A^{1/2}  }
% \bu^N
%\Vert_{H^1}^2
%+
\Vert
      e^{\sigma  A^{1/2}  }
 \bu^N
\Vert_{H^1}^6
\big) \ds
\\
&\leq
c_{\nu,\tilde{M},M}
\mathbb{E}
\int_0^{t }
 \ds
 \\
&\leq
c_{\nu,\tilde{M},M}
t.
 \end{aligned}
\end{equation}
On the other hand, by Doob’s inequality and the assumption on the noise in Section \ref{sec:assumptions},
\begin{equation}
\begin{aligned}
\mathbb{P}&\bigg(
c(\nu)\sup_{\sigma\in[0,{t \wedge \mathcal{T}^M_N} ]}\bigg\vert\int_0^\sigma\sum_{k\geq 1}
 \big\langle
 A^{1/2}   e^{ s  A^{1/2}  }
  P^N P\bg_k(\bu^N)
,\, 
\\&\qquad\qquad\qquad\qquad  A^{1/2}   e^{ s  A^{1/2}  } \bu^N 
\big\rangle \dd W^k_s
\bigg\vert
> \frac{\tilde{M}}{3} \bigg)
\\&
\leq c_{\nu,\tilde{M}}
\mathbb{E}
\int_0^{t \wedge \mathcal{T}^M_N}
\Vert     e^{\sigma  A^{1/2}  } \bu^N  \Vert_{H^1}^2
\sum_{k\geq 1}
 \Vert
   e^{\sigma  A^{1/2}  }
  P^N P\bg_k(\bu^N)
\Vert_{H^1}^2 \ds
 \\
&\leq
c_{\nu,\tilde{M},M}
t.
 \end{aligned}
\end{equation}
Similarly, we obtain by using  the first estimate of Lemma \ref{lem:main} and the assumption on the transport noise in Section \ref{sec:assumptions} that
\begin{equation}
\begin{aligned}
\mathbb{P}\bigg(
c(\nu)\sup_ {\sigma\in[0,{t \wedge \mathcal{T}^M_N} ]}\bigg\vert&\int_0^\sigma\sum_{k\geq 1}
 \big\langle
 A^{1/2}   e^{ s  A^{1/2}  }
  P^N P((\bm{\xi}_k \cdot \nabla)\bu^N)
,\,   
\\&\qquad\qquad\qquad\qquad
A^{1/2}   e^{ s  A^{1/2}  } \bu^N 
\big\rangle \dd W^k_s
\bigg\vert
> \frac{\tilde{M}^2}{3} \bigg)
\\&\leq
c_{\nu,\tilde{M},M}
t.
 \end{aligned}
\end{equation}
We therefore conclude that
\begin{equation}
\begin{aligned}
\mathbb{P} \bigg(\sup_{\sigma\in[0,{t \wedge \mathcal{T}^M_N} ]}\Vert   e^{\sigma  A^{1/2}  } \bu^N(\sigma) \Vert^2_{H^1} 
&+
\nu  \int_0^{t \wedge \mathcal{T}^M_N}
\Vert      e^{\sigma  A^{1/2}  }
 \bu^N
\Vert_{H^2}^2 \ds
\\&>
\Vert     
 %e^{\varphi(0) A^{1/2}  }
 \bu^N_0 \Vert^2_{H^1} 
+ \tilde{M} \bigg)
\leq 
c_{\nu,\tilde{M},M}
t
 \end{aligned}
\end{equation}
an thus, \eqref{probZeroX} immediately follows.
\\ \\
Our next goal is to show \eqref{expecZeroX}. For this, we take $N,n\in \mathbb{N}$ with
 $N\geq n$.  
We now let  $\bu^N$ and $\bu^n$ be the two Galerkin solutions of \eqref{eq:galerkin} with corresponding data $P^N\bu_0$ and $P^n\bu_0$ respectively. Set $\bu^{Nn}=\bu^N-\bu^n$ so that $\bu^{Nn}$ solves 
\begin{equation}
\begin{aligned}
\label{eq:galerkinNn}
\dd \bu^{Nn}  &+[  P^N P ((\bu^N\cdot\nabla) \bu^N)- P^n P ((\bu^n\cdot\nabla) \bu^n)  -\nu  \Delta \bu^{Nn}   ]\dt 
    \\&
    = 
    \frac{1}{2}\sum_{k\geq1}\big[ P^NP((\bm{\xi}_k \cdot\nabla)P( \bm{\xi}_k \cdot\nabla) \bu^N) 
    -
     P^nP((\bm{\xi}_k \cdot\nabla)P( \bm{\xi}_k \cdot\nabla) \bu^n)
     \big] \dt
    \\&
    +
    \sum_{k\geq1}[P^NP\bg_k(\bu^N)-P^nP\bg_k(\bu^n)] \dd W_t^k
\\&
-\sum_{k\geq1}\big[P^N  P((\bm{\xi}_k \cdot\nabla )\bu^N)- P^n P((\bm{\xi}_k \cdot\nabla )\bu^n) \big]  \dd W_t^k    
    ,
\\
&\bu^{Nn}_0= (P^N-P^n)\bu_0.
\end{aligned}
\end{equation}
We have by the product rule,
\begin{equation}
\begin{aligned}
    \dd &(  A^{1/2}   e^{ t  A^{1/2} } \bu^{Nn})= 
     A  e^{ t  A^{1/2} } \bu^{Nn} \dt
   \\&
   -
    A^{1/2}   e^{ t  A^{1/2} }
   [ P^N P (\bu^N\cdot\nabla) \bu^N 
   -
   P^n P (\bu^n\cdot\nabla) \bu^n 
    +\nu   A  \bu^{Nn}  ]\dt 
   \\&
   +
    \frac{1}{2}\sum_{k\geq1} A^{1/2}   e^{ t  A^{1/2} }\big[ P^NP((\bm{\xi}_k \cdot\nabla )P(\bm{\xi}_k \cdot\nabla )\bu^N) 
    \\&\qquad\qquad\qquad\qquad
    -
     P^nP((\bm{\xi}_k \cdot\nabla )P(\bm{\xi}_k \cdot\nabla )\bu^n)
     \big] \dt
    \\&
   + \sum_{k\geq1} A^{1/2}   e^{ t  A^{1/2} }[P^NP\bg_k(\bu^N)-P^nP\bg_k(\bu^n)] \dd W_t^k
      \\&
   - \sum_{k\geq1} A^{1/2}   e^{ t  A^{1/2} }\big[P^N  P((\bm{\xi}_k \cdot\nabla )\bu^N)- P^n P((\bm{\xi}_k \cdot\nabla )\bu^n) \big]  \dd W_t^k
\end{aligned}
\end{equation}
and thus, it follows that
{ 
\begin{equation}
\begin{aligned}
\label{diffEqnIto}
\frac{1}{2} &\Vert    e^{t  A^{1/2}  } \bu^{Nn}(t) \Vert^2_{H^1} 
+
\nu 
\int_0^t \Vert e^{\sigma  A^{1/2}  } \bu^{Nn} \Vert^2_{H^2}  \ds
= 
\frac{1}{2} \Vert    
%e^{\varphi(0) A^{1/2}  }
\bu^{Nn}_0 \Vert^2_{H^1} 
\\
&+
\int_0^t \big\langle
  A   e^{\sigma  A^{1/2}  } \bu^{Nn}
,\,   A^{1/2}   e^{\sigma  A^{1/2}  } \bu^{Nn} 
\big\rangle \ds
\\
&-
\int_0^t \big\langle
 A^{1/2}   e^{\sigma  A^{1/2}  }
    [ P^N P (\bu^N\cdot\nabla) \bu^N 
   \\&\qquad\qquad\qquad\qquad
   -
   P^n P (\bu^n\cdot\nabla) \bu^n]
,\,   A^{1/2}   e^{\sigma  A^{1/2}  } \bu^{Nn} 
\big\rangle \ds
\\
&+
\mathcal{C}_1((\mathbf{g}_k,\bm{\xi}_k,\bu^n,\bu^N,\bu^{Nn})(t)
% \\&
%   - \frac{1}{2}\sum_{k\geq1}\Vert     e^{ t  A^{1/2} }\big[P^N  P((\bm{\xi}_k \cdot\nabla )\bu^N)- P^n P((\bm{\xi}_k \cdot\nabla )\bu^n) \big] \Vert_{H^1}^2\ds
+\mathcal{C}_2(\mathbf{g}_k,\bm{\xi}_k,\bu^n,\bu^N,\bu^{Nn})(t)
\\
&+\int_0^t\sum_{k\geq 1}
 \big\langle
 A^{1/2}   e^{\sigma  A^{1/2}  }
   [P^NP\bg_k(\bu^N)-P^nP\bg_k(\bu^n)]
,\,  
\\&\qquad\qquad\qquad\qquad
 A^{1/2}   e^{\sigma  A^{1/2}  } \bu^{Nn} 
\big\rangle \dd W^k_\sigma
\\
&-\int_0^t\sum_{k\geq 1}
 \big\langle
 A^{1/2}   e^{\sigma  A^{1/2}  }
   \big[P^N  P((\bm{\xi}_k \cdot\nabla )\bu^N)- P^n P((\bm{\xi}_k \cdot\nabla )\bu^n) \big]
,\,  
\\&\qquad\qquad\qquad\qquad
 A^{1/2}   e^{\sigma  A^{1/2}  } \bu^{Nn} 
\big\rangle \dd W^k_\sigma
 \end{aligned}
\end{equation} 
where
\begin{align*}
\mathcal{C}_1&(\mathbf{g}_k,\bm{\xi}_k,\bu^n,\bu^N,\bu^{Nn})(t)
\\&:=
\frac{1}{2}\int_0^t
\sum_{k\geq 1}
\big\Vert
   e^{\sigma  A^{1/2}  }
   [P^NP\bg_k(\bu^N)-P^nP\bg_k(\bu^n)]
   \\&
   \qquad
   -
   e^{ \sigma  A^{1/2} }\big[P^N  P((\bm{\xi}_k \cdot\nabla )\bu^N)- P^n P((\bm{\xi}_k \cdot\nabla )\bu^n) \big]
\big\Vert_{H^1}^2 \ds
\end{align*}
is the It\^o correction term and
\begin{align*}
\mathcal{C}_2&(\mathbf{g}_k,\bm{\xi}_k,\bu^n,\bu^N,\bu^{Nn})(t)
\\&:=
\frac{1}{2}\int_0^t\sum_{k\geq 1}
 \big\langle
 A^{1/2}   e^{\sigma  A^{1/2}  }
   P^N P[(\bm{\xi}_k \cdot \nabla) P(\bm{\xi}_k \cdot \nabla)\bu^N]
,\,   A^{1/2}   e^{\sigma  A^{1/2}  } \bu^{Nn} 
\big\rangle \ds
\\
&\quad-\frac{1}{2}\int_0^t\sum_{k\geq 1}
 \big\langle
 A^{1/2}   e^{\sigma  A^{1/2}  }
   P^n P[(\bm{\xi}_k \cdot \nabla) P(\bm{\xi}_k \cdot \nabla)\bu^n]
,\,   A^{1/2}   e^{\sigma  A^{1/2}  } \bu^{Nn} 
\big\rangle \ds
\end{align*}
is the Stratonovich corrector. }
For $\mathcal{T}^M_{Nn}:= \mathcal{T}^M_N \wedge \mathcal{T}^M_n$,
we have that
\begin{equation}
\begin{aligned}
\bigg\vert&\int_0^{\mathcal{T}^M_{Nn}} \big\langle
  A   e^{t  A^{1/2}  } \bu^{Nn} 
,\,   A^{1/2}   e^{t  A^{1/2}  } \bu^{Nn} 
\big\rangle \dt \bigg\vert
\\&\leq 
\frac{\nu}{8}  \int_0^{\mathcal{T}^M_{Nn}}
\Vert     e^{t  A^{1/2}  }
 \bu^{Nn}
\Vert_{H^2}^2 \dt
+ c(\nu)
\int_0^{\mathcal{T}^M_{Nn}}
\Vert     e^{t  A^{1/2}  }
 \bu^{Nn}
\Vert_{H^1}^2 \dt.
\end{aligned}
\end{equation}
In order to estimate the nonlinear convective term, we rewrite it as follows
\begin{equation}
\begin{aligned}
\label{differConvec}
\big\langle
 A^{1/2}   e^{ t   A^{1/2}  }
    [ P^N P (\bu^N\cdot\nabla) \bu^N 
   &-
   P^n P (\bu^n\cdot\nabla) \bu^n]
,\,   A^{1/2}   e^{ t   A^{1/2}  } \bu^{Nn} 
\big\rangle
\\&= I_1+I_2+I_3
 \end{aligned}
\end{equation}
where
\begin{align*}
&I_1:= \big\langle
 A^{1/2}   e^{ t   A^{1/2}  }
    [ P^N P (\bu^{Nn}\cdot\nabla) \bu^N 
   ]
,\,   A^{1/2}   e^{ t   A^{1/2}  } \bu^{Nn} 
\big\rangle,
\\
&I_2:= \big\langle
 A^{1/2}   e^{ t   A^{1/2}  }
    [ P^N P (\bu^n\cdot\nabla) \bu^{Nn} 
   ]
,\,   A^{1/2}   e^{ t   A^{1/2}  } \bu^{Nn}
\big\rangle,
\\
&I_3:= \big\langle
 A^{1/2}   e^{ t   A^{1/2}  }
    [ (P^N-P^n) P (\bu^n\cdot\nabla) \bu^n 
  ]
,\,   A^{1/2}   e^{ t   A^{1/2}  } \bu^{Nn} 
\big\rangle.
\end{align*}
By using Lemma \ref{lem:main}, we obtain
\begin{align}
\vert I_1 \vert
&\lesssim
\Vert  A^{1/2}  e^{  t   A^{1/2}  }\bu^{Nn} \Vert_{L^2}^\frac{1}{2}
\Vert  A  e^{  t   A^{1/2}  }\bu^{Nn} \Vert_{L^2}^\frac{1}{2}
\Vert  A^{1/2}  e^{  t   A^{1/2}  }\bu^N \Vert_{L^2}
\Vert  A  e^{  t   A^{1/2}  }\bu^{Nn} \Vert_{L^2}
\nonumber
\\&
\leq 
\frac{\nu}{8}
\Vert     e^{  t   A^{1/2}  }\bu^{Nn} \Vert_{H^2}^2
+
c(\nu)
\Vert    e^{  t   A^{1/2}  }\bu^{Nn} \Vert_{H^1}^2
\Vert   e^{  t   A^{1/2}  }\bu^N \Vert_{H^1}^4.
\label{ioneNn}
\end{align}
Next,
\begin{equation}
\begin{aligned}
\label{itwoNn}
\vert I_2 \vert
&\lesssim
\Vert  A^{1/2}  e^{  t   A^{1/2}  }\bu^n \Vert_{L^2}^\frac{1}{2}
\Vert  A  e^{  t   A^{1/2}  }\bu^n \Vert_{L^2}^\frac{1}{2}
\Vert  A^{1/2}  e^{  t   A^{1/2}  }\bu^{Nn} \Vert_{L^2}
\Vert  A  e^{  t   A^{1/2}  }\bu^{Nn} \Vert_{L^2}
\\&
\leq 
\frac{\nu}{8}
\Vert     e^{  t   A^{1/2}  }\bu^{Nn} \Vert_{H^2}^2
+
c(\nu)
\Vert     e^{  t   A^{1/2}  }\bu^n \Vert_{H^2}^2
\Vert   e^{  t   A^{1/2}  }\bu^{Nn} \Vert_{H^1}^2.
\end{aligned}
\end{equation}
Finally, if we set $Q^n= I-P^n$ as the projection onto the modes higher than $n$,  then we can use the identity $P^N-P^n=Q^nP^N$ which holds for $N\geq n$ to obtain 
\begin{equation}
\begin{aligned}
\label{ithreeNn}
\vert I_3 \vert &=\big\vert \big\langle
e^{ t   A^{1/2}  }
    Q^n P^N P (\bu^n\cdot\nabla) \bu^n 
,\,   A   e^{ t   A^{1/2}  } \bu^{Nn} 
\big\rangle \big\vert
\\&
\leq \frac{\nu}{8}
\Vert      e^{ t   A^{1/2}  } \bu^{Nn} 
\Vert_{H^2}^2
+c(\nu)
\Vert  
    Q^n P^N e^{ t   A^{1/2}  } P (\bu^n\cdot\nabla) \bu^n 
\Vert_{L^2}^2.
\end{aligned}
\end{equation}
However, by employing the second   inequality in \eqref{poinare}, the continuity of $P^N$ shown in \eqref{continuityProperty} and  Lemma \ref{lem:h1convectiveX}, we obtain
\begin{equation}
\begin{aligned}
\Vert  &
    Q^n P^N e^{ t  A^{1/2} } P (\bu^n\cdot\nabla) \bu^n 
\Vert_{L^2}^2 
\leq
\frac{2}{n^2}
\Vert    e^{ t  A^{1/2} } P (\bu^n\cdot\nabla) \bu^n 
\Vert_{H^1}^2
\\&
\lesssim
\frac{1}{n^2}\Big( 
\Vert  e^{ t  A^{1/2}}\bu^n  \Vert_{L^2}^2
\Vert  e^{ t  A^{1/2}} \bu^n \Vert_{H^2}^2 
+
\Vert  e^{ t  A^{1/2}} \bu^n \Vert_{H^1}^2
\Vert e^{ t  A^{1/2}}  \bu^n  \Vert_{H^1}^2
\Big)
\\&
\lesssim
\frac{1}{n^2} 
\Vert  e^{ t  A^{1/2}}\bu^n  \Vert_{H^1}^2
\Vert  e^{ t  A^{1/2}} \bu^n \Vert_{H^2}^2  
%\\&
%\lesssim
%\frac{n}{n^2}\bigg(
%\Vert  \bu^n \Vert_{H^1}^2
%\Vert  e^{ t  A^{1/2}} \bu^n \Vert_{H^2}^2 
%+
%\Vert  e^{ t  A^{1/2}} \bu^n \Vert_{H^1}^2
%\Vert  \bu^n \Vert_{H^2}^2 
%\bigg)
\end{aligned}
\end{equation}
where the last inequality follows from the first estimate in \eqref{poinare}. Thus, it follows that
\begin{equation}
\begin{aligned}
\vert I_3 \vert 
\leq \frac{\nu}{8}
\Vert      e^{ t   A^{1/2}  } \bu^{Nn} 
\Vert_{H^2}^2
+\frac{c(\nu)}{n}\Vert  e^{ t  A^{1/2}}\bu^n  \Vert_{H^1}^2
\Vert  e^{ t  A^{1/2}} \bu^n \Vert_{H^2}^2 
\end{aligned}
\end{equation}
and as such,
\begin{equation}
\begin{aligned}
\bigg\vert&\int_0^{\mathcal{T}^M_{Nn}} \big\langle
 A^{1/2}   e^{ t   A^{1/2}  }
    [ P^N P (\bu^N\cdot\nabla) \bu^N 
   -
   P^n P (\bu^n\cdot\nabla) \bu^n]
,\,   A^{1/2}   e^{ t   A^{1/2}  } \bu^{Nn} 
\big\rangle \dt \bigg\vert
\\&\leq 
\frac{3\nu}{8}  \int_0^{\mathcal{T}^M_{Nn}}
\Vert     e^{ t   A^{1/2}  }
 \bu^{Nn}
\Vert_{H^2}^2 \dt
+ c(\nu)
\int_0^{\mathcal{T}^M_{Nn}}
\frac{1}{n}\Vert  e^{ t  A^{1/2}}\bu^n  \Vert_{H^1}^2
\Vert  e^{ t  A^{1/2}} \bu^n \Vert_{H^2}^2  \dt
\\&+ c(\nu)
\int_0^{\mathcal{T}^M_{Nn}}
\Vert    e^{ t   A^{1/2}  }\bu^{Nn} \Vert_{H^1}^2
\Big(
\Vert   e^{ t   A^{1/2}  }\bu^N \Vert_{H^1}^4 
+ 
\Vert     e^{ t   A^{1/2}  }\bu^n \Vert_{H^2}^2
\Big) \dt.
\end{aligned}
\end{equation}
Next, we use the Polarization Identity and \eqref{orthoTwoNoise} to obtain
\begin{equation}
\begin{aligned}
\label{triangleNoise}
\big\Vert
   e^{ t   A^{1/2}  }
   &[P^NP\bg_k(\bu^N)-P^nP\bg_k(\bu^n)]
   \\&-
   e^{  t   A^{1/2} }\big[P^N  P((\bm{\xi}_k \cdot\nabla )\bu^N)- P^n P((\bm{\xi}_k \cdot\nabla )\bu^n) \big]
\big\Vert_{H^1}^2
\\&=
\big\Vert
   e^{ t   A^{1/2}  }
   [P^NP\bg_k(\bu^N)-P^nP\bg_k(\bu^n)]
\big\Vert_{H^1}^2
\\&+
\big\Vert
   e^{  t   A^{1/2} }\big[P^N  P((\bm{\xi}_k \cdot\nabla )\bu^N)- P^n P((\bm{\xi}_k \cdot\nabla )\bu^n) \big]
\big\Vert_{H^1}^2.
\end{aligned}
\end{equation}
To estimate the first term on the right of \eqref{triangleNoise}, we use the triangle inequality again to obtain
\begin{equation}
\begin{aligned}
&  
\big\Vert
   e^{ t   A^{1/2}  }
   [P^NP\bg_k(\bu^N)-P^nP\bg_k(\bu^n)]
\big\Vert_{H^1}^2  
\\&\lesssim
\big\Vert
 P^NP  e^{ t   A^{1/2}  }
   [\bg_k(\bu^N)-\bg_k(\bu^n)]
\big\Vert_{H^1}^2  
+
\big\Vert
 Q^nP^N P e^{ t   A^{1/2}  }
   \bg_k(\bu^n)
\big\Vert_{H^1}^2  
\end{aligned}
\end{equation}
where $Q^n= I-P^n$. Thus, we obtain by using the assumptions on $(\bg_k)_{k\geq1}$ in Section \ref{sec:assumptions}, the second  inequality in \eqref{poinare} and the continuity of $P^NP$ that,
\begin{equation}
\begin{aligned}
\label{multinoiseNn}
\bigg\vert\int_0^{\mathcal{T}^M_{Nn}} & 
\sum_{k\geq 1}
\big\Vert
   e^{ t   A^{1/2}  }
   [P^NP\bg_k(\bu^N)-P^nP\bg_k(\bu^n)]
\big\Vert_{H^1}^2 \dt  \bigg\vert
\\&\lesssim 
\int_0^{\mathcal{T}^M_{Nn}}
\bigg(
\Vert   e^{ t   A^{1/2}  }\bu^{Nn} \Vert_{H^1}^2
+
\frac{1}{n^2}\Big(1+
\Vert   e^{ t   A^{1/2}  }\bu^n \Vert_{H^2}^2\Big)
\bigg) \dt.
\end{aligned}
\end{equation}
We now combine the treatment of the last term in \eqref{triangleNoise} with the $\mathcal{C}_2$-term in \eqref{diffEqnIto}. In particular, we note from \eqref{projectionIdentity} and \eqref{poinare} that
\begin{equation}
\begin{aligned}
J_1^k&:=
\big\Vert
   e^{  t   A^{1/2} }\big[P^N  P((\bm{\xi}_k \cdot\nabla )\bu^N)- P^n P((\bm{\xi}_k \cdot\nabla )\bu^n) \big]
\big\Vert_{H^1}^2
\\&
=
\big\Vert
   e^{  t   A^{1/2} }Q^nP^N  P(\bm{\xi}_k \cdot\nabla \bu^N)
\big\Vert_{H^1}^2
+
\big\Vert
   e^{  t   A^{1/2} }P^n  P(\bm{\xi}_k \cdot\nabla \bu^{Nn})
\big\Vert_{H^1}^2
\\&
\leq
\frac{1}{n^2}
\big\Vert
   e^{  t   A^{1/2} }P^N  P(\bm{\xi}_k \cdot\nabla \bu^N)
\big\Vert_{H^2}^2
+
\big\Vert
   e^{  t   A^{1/2} }P^n  P(\bm{\xi}_k \cdot\nabla \bu^{Nn})
\big\Vert_{H^1}^2
\end{aligned}
\end{equation}
and
\begin{equation}
\begin{aligned}
J_2^k&:= 
 \big\langle
 A^{1/2}   e^{ t   A^{1/2}  }
   P^N P[(\bm{\xi}_k \cdot \nabla) P(\bm{\xi}_k \cdot \nabla)\bu^N]
,\,   A^{1/2}   e^{ t   A^{1/2}  } \bu^{Nn} 
\big\rangle,
\\
J_3^k&:=-
 \big\langle
 A^{1/2}   e^{ t   A^{1/2}  }
   P^n P[(\bm{\xi}_k \cdot \nabla) P(\bm{\xi}_k \cdot \nabla)\bu^n]
,\,   A^{1/2}   e^{ t   A^{1/2}  } \bu^{Nn} 
\big\rangle
\end{aligned}
\end{equation}
are such that
\begin{equation}
\begin{aligned}
J_2^k+J_3^k
&:=
 \big\langle
 A^{1/2}   e^{ t   A^{1/2}  }
   Q^nP^N P[(\bm{\xi}_k \cdot \nabla) P(\bm{\xi}_k \cdot \nabla)\bu^N]
,\,   A^{1/2}   e^{ t   A^{1/2}  } \bu^{N} 
\big\rangle 
\\
&+
 \big\langle
 A^{1/2}   e^{ t   A^{1/2}  }
   P^n P[(\bm{\xi}_k \cdot \nabla) P(\bm{\xi}_k \cdot \nabla)\bu^{Nn}]
,\,   A^{1/2}   e^{ t   A^{1/2}  } \bu^{Nn} 
\big\rangle
\\
&-
 \big\langle
 A^{1/2}   e^{ t   A^{1/2}  }
   Q^nP^N P[(\bm{\xi}_k \cdot \nabla) P(\bm{\xi}_k \cdot \nabla)\bu^N]
,\,   A^{1/2}   e^{ t   A^{1/2}  } \bu^{n} 
\big\rangle.
\end{aligned}
\end{equation}
Now note that since $\bu^n=P^n\bu^n$, the last term above is
\begin{equation}
\begin{aligned}
 \big\langle
  Q^nA^{1/2}   e^{ t   A^{1/2}  }
  P^N P[(\bm{\xi}_k \cdot \nabla) P(\bm{\xi}_k \cdot \nabla)\bu^N]
,\,  P^n A^{1/2}   e^{ t   A^{1/2}  } \bu^{n} 
\big\rangle
=0.
\end{aligned}
\end{equation}
On the other hand, since $Q^n=Q^nQ^n$  is a projection, it follows from the second estimate in \eqref{poinare} that
\begin{equation}
\begin{aligned}
 \big\langle&
 A^{1/2}   e^{ t   A^{1/2}  }
   Q^nP^N P[(\bm{\xi}_k \cdot \nabla) P(\bm{\xi}_k \cdot \nabla)\bu^N]
,\,   A^{1/2}   e^{ t   A^{1/2}  } \bu^{N} 
\big\rangle 
%\\
%&=
% \big\langle
% Q^nQ^n A^{1/2}   e^{ t   A^{1/2}  }
%   P^N P[(\bm{\xi}_k \cdot \nabla) (\bm{\xi}_k \cdot \nabla)\bu^N]
%,\,   A^{1/2}   e^{ t   A^{1/2}  } \bu^{N} 
%\big\rangle 
%\\
%&=
% \big\langle
%Q^n A^{1/2}   e^{ t   A^{1/2}  }
%   P^N P[(\bm{\xi}_k \cdot \nabla) (\bm{\xi}_k \cdot \nabla)\bu^N]
%,\,   Q^nA^{1/2}   e^{ t   A^{1/2}  } \bu^{N} 
%\big\rangle 
\\
&\leq
\frac{1}{n^2}
 \big\langle
 A    e^{ t   A^{1/2}  }
   P^N P[(\bm{\xi}_k \cdot \nabla) P(\bm{\xi}_k \cdot \nabla)\bu^N]
,\,    A    e^{ t   A^{1/2}  } \bu^{N} 
\big\rangle.
\end{aligned}
\end{equation}
Therefore, by pairing each of the two terms estimating $J_1^k$ with the first two terms of the summand $J_2^k+J_3^k$, respectively, and using Lemma \ref{cor:xiAppendix}, keeping in mind the continuity property of  $P^N$, we obtain
\begin{equation}
\begin{aligned}
\vert J_1^k&+J_2^k+J_3^k \vert
\\&\lesssim 
\Big\vert
\big\Vert
   e^{  t   A^{1/2} }P^n  P((\bm{\xi}_k \cdot\nabla )\bu^{Nn})
\big\Vert_{H^1}^2
\\&+
 \big\langle
 A^{1/2}   e^{ t   A^{1/2}  }
   P^n P[(\bm{\xi}_k \cdot \nabla) P(\bm{\xi}_k \cdot \nabla)\bu^{Nn}]
,\,   A^{1/2}   e^{ t   A^{1/2}  } \bu^{Nn} 
\big\rangle
\Big\vert
\\&
+
\frac{1}{n^2}
\Big\vert
\big\Vert
   e^{  t   A^{1/2} }P^N  P((\bm{\xi}_k \cdot\nabla )\bu^N)
\big\Vert_{H^2}^2
\\&+
\big\langle
 A   e^{ t   A^{1/2}  }
   P^N P[(\bm{\xi}_k \cdot \nabla)P (\bm{\xi}_k \cdot \nabla)\bu^N]
,\,   A   e^{ t   A^{1/2}  } \bu^{N} 
\big\rangle 
\Big\vert
\\
&\lesssim 
\Vert
   e^{ t   A^{1/2}  }
    \bm{\xi}_k \Vert_{H^r}^2
\Big(
\Vert    e^{ t   A^{1/2}  } \bu^{Nn} 
\Vert_{H^1}^2
+
\frac{1}{n^2}
\Vert    e^{ t   A^{1/2}  } \bu^{N} 
\Vert_{H^2}^2
\Big)
\end{aligned}
\end{equation}
with $r>4$.
By the Burkholder--Davis--Gundy inequality and the same argument used in \eqref{multinoiseNn}, we obtain
\begin{equation}
\begin{aligned}
\mathbb{E} &\sup_{t\in[0,\mathcal{T}^M_{Nn}]}\bigg\vert\int_0^t 
\sum_{k\geq 1}
 \big\langle
 A^{1/2}   e^{\sigma  A^{1/2}  }
   [P^NP\bg_k(\bu^N)-P^nP\bg_k(\bu^n)]
,\,   
\\&\qquad\qquad\qquad\qquad
A^{1/2}   e^{\sigma  A^{1/2}  } \bu^{Nn} 
\big\rangle \dd W^k_\sigma  \bigg\vert
\\&\lesssim 
\mathbb{E}  \bigg(\int_0^{\mathcal{T}^M_{Nn}} 
\sum_{k\geq 1}
 \big\langle
 A^{1/2}   e^{t  A^{1/2}  }
   [P^NP\bg_k(\bu^N)-P^nP\bg_k(\bu^n)]
,\,   
\\&\qquad\qquad\qquad\qquad
A^{1/2}   e^{t  A^{1/2}  } \bu^{Nn} 
\big\rangle^2 \dt  \bigg)^\frac{1}{2}
\\&\lesssim 
\mathbb{E}  \bigg(\int_0^{\mathcal{T}^M_{Nn}} 
\sum_{k\geq 1}
 \big\Vert   e^{t  A^{1/2}  }
   [P^NP\bg_k(\bu^N)-PP^n\bg_k(\bu^n)]
\big\Vert_{H^1}^2 
\\&\qquad\qquad\qquad\qquad
\times\Vert   e^{t  A^{1/2}  } \bu^{Nn} 
\big\Vert_{H^1}^2 \dt  \bigg)^\frac{1}{2}
\\&\leq
\frac{1}{4} 
\mathbb{E}  \sup_{t\in[0,\mathcal{T}^M_{Nn}]}
\Vert   e^{t  A^{1/2}  } \bu^{Nn}(t) 
\big\Vert_{H^1}^2
\\&+
c\mathbb{E}
\int_0^{\mathcal{T}^M_{Nn}}
\bigg(
\Vert   e^{t  A^{1/2}  }\bu^{Nn} \Vert_{H^1}^2
+
\frac{1}{n^2}\Big(1+
\Vert   e^{t  A^{1/2}  }\bu^n \Vert_{H^2}^2\Big)
\bigg) \dt.
\end{aligned}
\end{equation}
It now remains to estimate the last term in \eqref{diffEqnIto}. For this, we need the following preliminary estimate. Similar to \eqref{differConvec}, we rewrite 
\begin{align*} 
\big\langle
 A^{1/2}   e^{ t   A^{1/2}  }
   \big[P^N  P((\bm{\xi}_k \cdot\nabla )\bu^N)&- P^n P((\bm{\xi}_k \cdot\nabla )\bu^n) \big]
,\,   A^{1/2}   e^{ t   A^{1/2}  } \bu^{Nn} 
\big\rangle
\\&= K_1^k+K_2^k 
 \end{align*}
where
\begin{align*}
&K_1^k:= \big\langle
 A^{1/2}   e^{ t   A^{1/2}  }
    [ P^N P (\bm{\xi}_k\cdot\nabla) \bu^{Nn} 
   ]
,\,   A^{1/2}   e^{ t   A^{1/2}  } \bu^{Nn}
\big\rangle,
\\
&K_2^k:= \big\langle
 A^{1/2}   e^{ t   A^{1/2}  }
    [ (P^N-P^n) P (\bm{\xi}_k\cdot\nabla) \bu^n 
  ]
,\,   A^{1/2}   e^{ t   A^{1/2}  } \bu^{Nn} 
\big\rangle.
\end{align*}
Since $(\bm{\xi})_{k\geq1}$ satisfies \eqref{xikbounded} by assumption, it follows by using Lemma \ref{lem:h1convectiveX} that 
\begin{equation}
\begin{aligned}
\sum_{k\geq 1}\vert K_1^k \vert
&\lesssim
%\sum_{k\geq 1}\Vert  A^{1/2}  e^{  t   A^{1/2}  }\bm{\xi}_k \Vert_{L^2}^\frac{1}{2}
%\Vert  A  e^{  t   A^{1/2}  }\bm{\xi}_k \Vert_{L^2}^\frac{1}{2}
%\Vert  A^{1/2}  e^{  t   A^{1/2}  }\bu^{Nn} \Vert_{L^2}
%\Vert  A  e^{  t   A^{1/2}  }\bu^{Nn} \Vert_{L^2}
%\\&\lesssim
\Vert   e^{  t   A^{1/2}  }\bu^{Nn} \Vert_{H^1}\Vert   e^{  t   A^{1/2}  }\bu^{Nn} \Vert_{H^2}
\end{aligned}
\end{equation}
and thus, by the Burkholder--Davis--Gundy inequality,
\begin{equation}
\begin{aligned}
\label{bdgtransportNoise}
\mathbb{E} &\sup_{t\in[0,\mathcal{T}^M_{Nn}]}\bigg\vert\int_0^t 
\sum_{k\geq 1}
  K_1^k   \dd W^k_\sigma  \bigg\vert
 \lesssim
 \mathbb{E} \bigg(\int_0^{\mathcal{T}^M_{Nn}} 
\sum_{k\geq 1}
 \vert K_1^k \vert^2 \dt  \bigg)^\frac{1}{2}
\\&
\leq\frac{c \,\epsilon}{2}
\mathbb{E}  \sup_{t\in[0, \mathcal{T}^M_{Nn}] } 
\Vert   e^{ t  A^{1/2}  }\bu^{Nn} (t)\Vert_{H^1}^2 
+
\frac{c}{2\epsilon}
\mathbb{E}   \int_0^{\mathcal{T}^M_{Nn}} 
\Vert   e^{ t  A^{1/2}  }\bu^{Nn} \Vert_{H^2}^2 \dt  
\end{aligned}
\end{equation}
for some $\epsilon>0$.
On the other hand, since $P^N-P^n=Q^nP^N$  is a continuous operator \eqref{continuityProperty}, it follows from  Lemma \ref{lem:h1convectiveX}  and \eqref{xikbounded} that  
\begin{align*}
\sum_{k\geq 1}\vert K_2^k \vert 
&\lesssim
%\Vert      e^{\sigma  A^{1/2}  } \bu^{Nn} 
%\Vert_{H^1}
%\Vert  
%     e^{\sigma  A^{1/2}  } P (\bm{\xi}_k\cdot\nabla) \bu^n 
%\Vert_{H^1}
%\\
%&\lesssim  
\Vert      e^{t  A^{1/2}  } \bu^{Nn} 
\Vert_{H^1}
\Vert  
     e^{t  A^{1/2}  }  \bu^n 
\Vert_{H^2}.
\end{align*}
Thus, by the Burkholder--Davis--Gundy inequality,
%{ 
\begin{equation}
\begin{aligned}
\label{bdgtransportNoise2}
\mathbb{E} &\sup_{t\in[0,\mathcal{T}^M_{Nn}]}\bigg\vert\int_0^t 
\sum_{k\geq 1}
  K_2^k   \dd W^k_\sigma  \bigg\vert
% \lesssim
%\mathbb{E}  \bigg( \int_0^{\mathcal{T}^M_{Nn}} 
%\Vert   e^{t  A^{1/2}  }\bu^{Nn} \Vert_{H^1}^2\Vert   e^{t  A^{1/2}  }\bu^n \Vert_{H^2}^2 \dt\bigg)^\frac{1}{2}
\\&
 \lesssim
\mathbb{E}  \int_0^{\mathcal{T}^M_{Nn}} 
\Vert   e^{t  A^{1/2}  }\bu^{Nn}  \Vert_{H^1}^2\big(1+\Vert   e^{t  A^{1/2}  }\bu^n \Vert_{H^2}^2\big) \dt.
\end{aligned}
\end{equation}
%}
From \eqref{bdgtransportNoise} and \eqref{bdgtransportNoise2}, it follows   that
\begin{equation}
\begin{aligned}
\label{bdgtransportNoise3}
\mathbb{E} &\sup_{t\in[0,\mathcal{T}^M_{Nn}]}\bigg\vert\int_0^t 
\sum_{k\geq 1}
 \big\langle
 A^{1/2}   e^{\sigma  A^{1/2}  }
   \big[P^N  P((\bm{\xi}_k \cdot\nabla )\bu^N)
   \\&- P^n P((\bm{\xi}_k \cdot\nabla )\bu^n) \big]
,\,   A^{1/2}   e^{\sigma  A^{1/2}  } \bu^{Nn} 
\big\rangle \dd W^k_\sigma  \bigg\vert
\\&
\leq\frac{c \,\epsilon}{2}
\mathbb{E}  \sup_{t\in[0, \mathcal{T}^M_{Nn}] } 
\Vert   e^{t  A^{1/2}  }\bu^{Nn}(t) \Vert_{H^1}^2 
+
\frac{c}{2\epsilon}
\mathbb{E}   \int_0^{\mathcal{T}^M_{Nn}} 
\Vert   e^{ t  A^{1/2}  }\bu^{Nn} \Vert_{H^2}^2 \dt 
\\&+
c(\nu)
\mathbb{E}   \int_0^{\mathcal{T}^M_{Nn}} 
\Vert   e^{t  A^{1/2}  }\bu^{Nn} \Vert_{H^1}^2\big(1+\Vert   e^{t A^{1/2}  }\bu^n \Vert_{H^2}^2 \big) \dt.
\end{aligned}
\end{equation}
By collecting the various estimates of \eqref{diffEqnIto}, we obtain
\begin{equation}
\begin{aligned}
\label{finalDiffEst}
\kappa_1&
\mathbb{E}  \sup_{t\in[0, \mathcal{T}^M_{Nn}] } 
\Vert   e^{ t  A^{1/2}  }\bu^{Nn}(t) \Vert_{H^1}^2  
+
\kappa_2
\mathbb{E}
\int_0^{\mathcal{T}^M_{Nn}} \Vert e^{t  A^{1/2}  } \bu^{Nn} \Vert^2_{H^2}  \dt
\leq
\frac{1}{2}
\mathbb{E} \Vert    
%e^{\varphi(0) A^{1/2}  }
\bu^{Nn}_0 \Vert^2_{H^1} 
\\
&+ c(\nu)\mathbb{E}
\int_0^{\mathcal{T}^M_{Nn}}
\Vert  e^{t A^{1/2}} \bu^{Nn} \Vert_{H^1}^2 
\Big(1
+
\Vert  e^{t A^{1/2}} \bu^N \Vert_{H^1}^4 
+
\Vert  e^{t A^{1/2}} \bu^n \Vert_{H^2}^2 
\Big) \dt
\\
&+ c(\nu)\mathbb{E}
\int_0^{\mathcal{T}^M_{Nn}}
\frac{1}{n}\Big(
1+
\Vert  e^{t A^{1/2}} \bu^n \Vert_{H^2}^2 
%+
%\Vert  e^{t A^{1/2}} \bu^N \Vert_{H^2}^2 
+
\Vert  e^{t A^{1/2}} \bu^n \Vert_{H^1}^2
\Vert  e^{t A^{1/2}} \bu^n \Vert_{H^2}^2 
\Big) \dt
 \end{aligned}
\end{equation}
where $\kappa_1:=\frac{1}{4}-\frac{c\epsilon}{2}$ and $\kappa_2=\frac{\nu}{2} 
-
\frac{c}{2\epsilon}$. 
We can now choose $\epsilon=\epsilon(\nu)\in(c/\nu, 1/2c)$ (recall that $c=c(\nu)$) appropriately so that both $\kappa_1$ and $\kappa_2$ are strictly positive. Since all constants are independent of $n$ or $N$, our required convergence \eqref{expecZeroX} follows from Gr\"onwall's lemma \cite[Lemma 5.3]{glatt2009strong}.
\end{proof}
\noindent With Lemma \ref{lem:criteria} in hand, our next goal is to apply Lemma \ref{lem:comparison} to get our final result that leads to the proof of Theorem \ref{thm:main}.

\begin{proposition}
\label{prop:local}
%Let assumptions in Section \ref{sec:assumptions} hold for $r=s=1$ and $\varphi =T$. 
Let $(\bg_k)_{k\geq1}$ satisfy \eqref{noiseGrowth}--\eqref{noiseMeanfree} with  $r=s=1$  and $\sigma_1=t $ and let $(\bm{\xi}_k)_{k\geq1}$ satisfy \eqref{xikbounded}  with $r>4$, $s=1$ and $\sigma_2=T$.
If  there exists a deterministic $K_0>0$ such that \eqref{initialCondBounded} holds,
then $\bu(t \wedge \tau )\in  C((0,\tau ); D(e^{t  A^{1/2s}} : \mathbb{H}^r(\mathbb{T}^d)))$ a.s.
for all $t\in (0, T]$.
\end{proposition}
\begin{proof}
We note that the projection of the initial condition satisfying $\Vert   \bu_0 \Vert_{H^1} \leq K_0$ a.s. will also satisfies $\Vert   \bu^N_0 \Vert_{H^1} \leq K_0$  a.s. Now take the associated family $\{\bu^N\}_{N\in \mathbb{N}}$ family of Galerkin solution. Due to Lemma \ref{lem:criteria}, we see that the assumptions of Lemma \ref{lem:comparison} are satisfied and as such, there exists a pair $(\bu,\tau)$ with $\tau>0$ a.s. and
\begin{align*}
\bu(\cdot)=\bu(\cdot \wedge \tau) \in C\big([0,\tau );D(e^{ \cdot   A^{1/2}} : \mathbb{H}^1(\mathbb{T}^d))\big) \cap L^2\big((0,\tau );D(e^{ \cdot   A^{1/2}} : \mathbb{H}^{2}(\mathbb{T}^d))\big)
\end{align*}
such that up to taking a subsequence (not relabelled)
\begin{align*}
\sup_{\sigma\in[0,{  \tau} ]}\Vert   e^{\sigma A^{1/2}} [\bu^N(\sigma) -\bu(\sigma)]\Vert^2_{H^1} 
&+
\nu
\int_0^{ \tau}
\Vert   e^{\sigma A^{1/2}}
 [\bu^N-\bu]
\Vert_{H^{2}}^2 \ds \rightarrow 0 \quad \text{a.s.}
\end{align*}
\end{proof}

\subsection{Identification of the limit}  
Note that in the definition of a Gevrey class solution, Definition \ref{def:GevreySol}, identification of the limit system has nothing to do with the Gevrey class regularity. Indeed, the limit equation is identified at the level of construction of a maximal strong pathwise solution $(\bu, \tau_m)$. We refer to \cite{glatt2009strong,  mikulevicius2009strong, mikulevicius2004stochastic} for  this procedure. Due to Theorem \ref{thm:globStrong}, it follows that $\tau$ from Proposition \ref{prop:local} is such that  $\tau\leq \tau_m$ a.s. since the latter is maximal. Nevertheless, similar to \cite[Section 4.2]{glatt2009strong}, we can use a contradiction argument to extend $\tau$ to a maximal stopping time $\tau_{\max}\leq \tau_m$ whose announcing times are given by an increasing sequence of stopping times $(\tau_l)_{l\geq1}$. However,{ it is an open question if $\tau_{\max}=\infty$ a.s. in 2D (but certainly not expected in 3D) and one main drawback to proving this result lies in the fact that when working in the Gevrey class,
\begin{align*}
\big\langle A^{1/2} e^{t A^{1/2}} P(\bu\cdot\nabla)\bu\,,\, A^{1/2} e^{t A^{1/2}}\bu \big\rangle \neq 0
\end{align*}
unlike the cancellation property \cite[Lemma 3.1]{temam1995navier}
\begin{align*}
\big\langle  A^{1/2}P(\bu\cdot\nabla)\bu\,,\, A^{1/2}\bu \big\rangle = 0.
\end{align*}
enjoyed when working in the usual Sobolev space defined on $\mathbb{T}^2$.
}

\section{Decay Rate}
\label{sec:decay}
\noindent In the section, we show  Theorem \ref{thm:decay} which gives the decay rate of the Fourier modes of the velocity field solving the Navier--Stokes equation under the assumption that $\tau=\infty$ a.s. To begin with, we recall $\bu$ and its finite-dimensional approximation $\bu^N$ that solves
\begin{equation}
\begin{aligned}
\label{bu}
&\dd \bu  +[ P ((\bu\cdot\nabla) \bu) + \nu  A \bu   ]\dt = \frac{1}{2} \sum_{k\geq1} P[(\bm{\xi}_k \cdot\nabla)P(\bm{\xi}_k \cdot\nabla )\bu] \dt
\\&\qquad\qquad\qquad\qquad\qquad\quad
+
    \sum_{k\geq1}P\big[\bg_k(\bu)-((\bm{\xi}_k \cdot\nabla )\bu)\big] \dd W_t^k,
\\
&\bu(\bx,0)= \bu_0(\bx)
\end{aligned}
\end{equation}
and
\begin{equation}
\begin{aligned}
\label{buN}
    \dd \bu^N  &+[P^N P ((\bu^N\cdot\nabla) \bu^N) + \nu  A \bu^N   ]\dt     
    \\&= \frac{1}{2} \sum_{k\geq1}P^N P[(\bm{\xi}_k \cdot\nabla ) P(\bm{\xi}_k \cdot\nabla )\bu^N] \dt
    \\&
    +
    \sum_{k\geq1}P^NP\big[\bg_k(\bu^N)-((\bm{\xi}_k \cdot\nabla )\bu^N)\big] \dd W_t^k,
    \\
    &\bu^N_0(\bx)= P^N\bu_0(\bx)
\end{aligned}
\end{equation}
respectively. Having shown the Gevrey class $1$ regularity of $\bu$ { and with the assumption $\tau=\infty$ in hand}, we wish to use this information to obtain the following exponential decay rate for the difference of $\bu$ and $\bu^N$ in the   $H^1$-norm (i.e. in the class of strong solutions).
Before we begin, we note that from Theorem \ref{thm:main}, we have that the estimate
\begin{align*}
\Vert \bu(t) \Vert_{G^1}^2
&=
 \sum_{\bk\in  \mathbb{Z}^d}  \vert \bk\vert^{2}e^{2t \vert \bk\vert}
  \vert\hat{\bu}_{\bk}(t)\vert^2
\lesssim 1  
%\leq c_7\lambda_1\nu^2 + \Vert \bu_0\Vert_{H^1}^2 + \frac{2}{\nu^2\lambda_1}\Vert \bff\Vert_{L^2}^2
\end{align*}
holds a.s. for all $t\in(0,T]$
with a constant depending only on $\nu$, 
 $\bu_0$ and $K$ from \eqref{xikbounded}. Thus, there exists $\delta_1\in(0,T)$ such that
\begin{align}
\label{eq:betterThanExp}
 &\vert\hat{\bu}_{\bk}(t)\vert^2
 \lesssim
    \frac{1}{\vert \bk\vert^2} e^{-2\delta_1 \vert \bk\vert},
  \qquad \qquad
  \vert \bk\vert^2\vert\hat{\bu}_{\bk}(t)\vert^2
  \lesssim
  e^{-2\delta_1 \vert \bk\vert}
\end{align}
holds a.s. for $t\in[\delta_1,T]$.  The first bound means that each Fourier coefficient $\hat{\bu}_{\bk}$ decays better-than-exponential, with respect to its wavenumber $\frac{1}{\vert \bk\vert}$, { uniformly with respect to $t\in [\delta_1,T]$}, uniformly with respect to the initial enstrophy $ \Vert \bu_0\Vert_{H^1}^2$ and uniformly with respect to the noise coefficients
$(\bg_k,\bm{\xi}_k)_{k\geq1}$. With this preliminary information in hand, we can now show Theorem \ref{thm:decay}. First of all, note that 
\begin{align}
\label{splitEst}
\Vert (\bu^N - \bu)(t) \Vert_{H^1} 
\leq
\Vert Q^N \bu(t)\Vert_{H^1} 
+
\Vert (P^N\bu - \bu^N)(t)\Vert_{H^1} 
\end{align}
where
\begin{align*}
\Vert Q^N \bu(t)\Vert_{H^1}^2 
=
\Vert (I-P^N) \bu(t)\Vert_{ H^1}^2 
=
\sum_{\vert \bk\vert  > N}    \vert \bk\vert^{2} \vert \hat{\bu}_{\bk}(t) \vert^2
\end{align*}
and so, the estimate
\begin{align}
\label{qNEst}
%\sup_{t\in(0,T)}
%\mathbb{E}
\Vert Q^N \bu(t)\Vert_{H^1}^2 
\lesssim
%&\leq
%\bigg( c_7\lambda_1\nu^2 + \Vert \bu_0\Vert_{H^1}^2 + \frac{2}{\nu^2\lambda_1}\Vert \bff\Vert_{L^2}^2\bigg)
%\bigg(\frac{1}{N^2}+1\bigg)
   e^{-2\delta_1 N}
\end{align}
holds a.s. for $t\in[\delta_1,T]$.
In order to obtain our desired decay, therefore, it remains to estimate $\Vert      P^N  \bu-\bu^N \Vert^2_{H^1}$.
For this, we first note that
\begin{align}
\label{solenoidalDifference}
\frac{1}{2} \Vert      P^N  \bu-\bu^N \Vert^2_{H^1}
= 
 \frac{1}{2} \Vert      \bu^N  \Vert^2_{H^1}
 +
 \frac{1}{2} \Vert      P^N  \bu  \Vert^2_{H^1}
- \big\langle  A^{1/2}P^N\bu \,,\, A^{1/2}\bu^N \big\rangle 
\end{align}
and as such, we can compute the right-hand side using It\^o's formula. Firstly, by applying It\^o's formula to the mapping $t\mapsto \frac{1}{2} \Vert A^{1/2}\bu^N(t) \Vert_{L^2}^2$, we obtain from \eqref{buN},
\begin{equation}
\begin{aligned}
\frac{1}{2}& \Vert  \bu^N(t) \Vert^2_{H^1} 
+
\nu \int_0^t \Vert 
    A  \bu^N \Vert_{L^2}^2 \ds
\\&= 
\frac{1}{2} \Vert       \bu^N_0 \Vert^2_{H^1} 
-
\int_0^t \big\langle
 A^{1/2}    
    P^NP [(\bu^N\cdot\nabla) \bu^N ]
\,,\,   A^{1/2}     \bu^N 
\big\rangle \ds  
\\
&+
\frac{1}{2}\int_0^t\sum_{k\geq 1}
 \big\langle
 A^{1/2}    
   P^N P[(\bm{\xi}_k \cdot \nabla)P(\bm{\xi}_k \cdot \nabla)\bu^N]
\,,\,   A^{1/2}     \bu^N 
\big\rangle \ds 
\\
&+
\frac{1}{2}\int_0^t
\sum_{k\geq 1}
\big\Vert    
   P^NP[\bg_k(\bu^N)
   -
   ((\bm{\xi}_k \cdot \nabla)\bu^N)]
\big\Vert_{H^1}^2 \ds
\\
&+\int_0^t\sum_{k\geq 1}
 \big\langle
 A^{1/2}    
   P^NP[\bg_k(\bu^N)
   -
   ((\bm{\xi}_k \cdot \nabla)\bu^N)]
\,,\,   A^{1/2}     \bu^N 
\big\rangle \dd W^k_\sigma.
 \end{aligned}
\end{equation}
Similarly, by applying It\^o's formula to the mapping $t\mapsto \frac{1}{2} \Vert A^{1/2}P^N  \bu (t)\Vert_{L^2}^2$, we obtain from \eqref{bu}
\begin{equation}
\begin{aligned}
\frac{1}{2}& \Vert      P^N  \bu(t) \Vert^2_{H^1} 
+
\nu \int_0^t \Vert 
    A  P^N  \bu
\Vert_{L^2}^2 \ds
\\&= 
\frac{1}{2} \Vert     P^N  \bu_0 \Vert^2_{H^1} 
-
\int_0^t \big\langle
 A^{1/2}    
    P^NP [(\bu\cdot\nabla) \bu ]
\,,\,   A^{1/2}    P^N  \bu 
\big\rangle \ds
\\
&+
\frac{1}{2}\int_0^t\sum_{k\geq 1}
 \big\langle
 A^{1/2}    
   P^N P[(\bm{\xi}_k \cdot \nabla) P(\bm{\xi}_k \cdot \nabla)\bu]
\,,\,  A^{1/2}    P^N  \bu
\big\rangle \ds
\\
&+
\frac{1}{2}\int_0^t
\sum_{k\geq 1}
\big\Vert    
   P^NP[\bg_k(\bu )
   -
   ((\bm{\xi}_k \cdot \nabla)\bu )]
\big\Vert_{H^1}^2 \ds
\\
&+\int_0^t\sum_{k\geq 1}
 \big\langle
 A^{1/2}    
   P^NP[\bg_k(\bu) - ((\bm{\xi}_k \cdot \nabla)\bu)]
\,,\,  A^{1/2}    P^N  \bu 
\big\rangle \dd W^k_\sigma.
 \end{aligned}
\end{equation}
Next, we use It\^o's product rule to obtain
{ 
\begin{equation}
\begin{aligned}
\big\langle &
A^{1/2}    P^N\bu \,,\, A^{1/2}\bu^N
\big\rangle  
\\&=
\big\langle 
A^{1/2}    P^N\bu_0 \,,\, A^{1/2}\bu^N_0
\big\rangle  
 -
 2\nu \int_0^t
\big\langle
A    \bu^N
\,,\, A     P^N \bu
\big\rangle \ds
\\&
-\int_0^t
\big\langle 
A^{1/2}    P^N P[(\bu\cdot \nabla)\bu]
\,,\, A^{1/2}\bu^N
\big\rangle \ds 
\\
&-\int_0^t
\big\langle
A^{1/2} P^N P [(\bu^N\cdot\nabla) \bu^N]
\,,\, A^{1/2}    P^N \bu
 \big\rangle \ds
\\&+\mathcal{C}(\mathbf{g}_k,\bm{\xi}_k,\bu,\bu^N)(t)
\\&+
\int_0^t
\sum_{k\geq1}
\big\langle
A^{1/2} P^NP\big[\bg_k(\bu^N)- ((\bm{\xi}_k \cdot\nabla )\bu^N) \big] 
\,,\, A^{1/2}    P^N\bu
\big\rangle \dd W_\sigma^k
\\&+
\int_0^t
\sum_{k\geq1}
\big\langle
A^{1/2}    P^NP[\bg_k(\bu) 
- ((\bm{\xi}_k\cdot \nabla)\bu)] 
\,,\, A^{1/2}\bu^N
\big\rangle \dd W_\sigma^k.
\end{aligned}
\end{equation} 
where
\begin{align*}
\mathcal{C}(\mathbf{g}_k,\bm{\xi}_k,\bu,\bu^N)(t)&:=
 \frac{1}{2}
\int_0^t\sum_{k\geq 1}
\big\langle
A^{1/2}    P^N P [(\bm{\xi}_k\cdot\nabla) P(\bm{\xi}_k\cdot\nabla) \bu]
\,,\,  A^{1/2}\bu^N
\big\rangle \ds 
\\
&+
\frac{1}{2}
\int_0^t\sum_{k\geq 1}
\big\langle
A^{1/2} P^N P [(\bm{\xi}_k\cdot\nabla) P (\bm{\xi}_k\cdot\nabla) \bu^N]
\,,\, A^{1/2}    P^N \bu
\big\rangle \ds
\\&+
\int_0^t
\sum_{k\geq1}
\big\langle
A^{1/2}    P^N P[\bg_k(\bu) 
- ((\bm{\xi}_k\cdot \nabla)\bu)] 
\,,\, 
\\&\qquad\qquad
A^{1/2} P^NP\big[\bg_k(\bu^N)- ((\bm{\xi}_k \cdot\nabla )\bu^N) \big] 
\big\rangle \ds.
\end{align*}  
}
Now notice that
\begin{align*}
2\nu
\big\langle
A    \bu^N
\,,\, A     P^N \bu
\big\rangle
- \nu
\Vert 
    A  P^N  \bu
\Vert_{L^2}^2
-
\nu
\Vert 
    A\bu^N
\Vert_{L^2}^2
=
-
\nu\Vert A(P^N\bu - \bu^N)\Vert_{L^2}^2.
\end{align*}
Next, by using the fact that $A^{1/2}$ is self-adjoint, we can verify that  
\begin{align*}
-&
\big\langle
 A^{1/2}    
    P^NP [(\bu^N\cdot\nabla) \bu^N]
\,,\,   A^{1/2}     \bu^N 
\big\rangle   
-
\big\langle
 A^{1/2}    
    P^NP [(\bu\cdot\nabla) \bu]
\,,\,   A^{1/2}    P^N  \bu 
\big\rangle  
\\&
+ 
\big\langle 
A^{1/2}    P^N P[(\bu\cdot \nabla)\bu]
\,,\, A^{1/2}\bu^N
\big\rangle  
+ 
\big\langle
A^{1/2} P^N P [(\bu^N\cdot\nabla) \bu^N]
\,,\, A^{1/2}    P^N \bu
 \big\rangle
% \\
%&=-
%\big\langle
% A^{1/2}    
%    P^NP (\bu^N\cdot\nabla) \bu^N
%\,,\,   A^{1/2}     (\bu^N - P^N  \bu )
%\big\rangle 
%\\&
%-
%\big\langle 
%A^{1/2}    P^N P((\bu\cdot \nabla)\bu)
%\,,\, A^{1/2} (P^N \bu -\bu^N)
%\big\rangle 
% \\
%&=
%\big\langle
%    P^NP (\bu^N\cdot\nabla) \bu^N
%\,,\,   A     (P^N  \bu - \bu^N)
%\big\rangle 
%\\&
%-
%\big\langle 
%  P^N P((\bu\cdot \nabla)\bu)
%\,,\, A  (P^N \bu -\bu^N)
%\big\rangle 
 \\
&=
\big\langle
    P^NP [ ((\bu^N- P^N  \bu)\cdot\nabla) \bu^N]
\,,\,   A     (P^N  \bu - \bu^N)
\big\rangle 
 \\
&-
\big\langle 
  P^N P[((\bu -P^N \bu)\cdot \nabla)\bu^N  ]
\,,\, A  (P^N \bu -\bu^N)
\big\rangle 
 \\
&+
\big\langle
    P^NP[ (  \bu \cdot\nabla) (\bu^N-\bu) ]
\,,\,   A     (P^N  \bu - \bu^N)
\big\rangle 
\\&
=:I_1+I_2+I_3
 \end{align*} 
where by the continuity property of $P^N$,
\begin{equation}
\begin{aligned}
\vert I_1\vert
&\leq 
\Vert A     (P^N  \bu - \bu^N)\Vert_{L^2}
\Vert  P^N  \bu - \bu^N \Vert_{L^4}
\Vert \nabla\bu^N\Vert_{L^4}
\\
&\leq 
\frac{\nu}{6}
\Vert A     (P^N  \bu - \bu^N)\Vert_{L^2}^2
+
c(\nu)
\Vert  P^N  \bu - \bu^N \Vert_{H^1}^2
\Vert \bu^N\Vert_{H^2}^2
\end{aligned}
\end{equation}
and for $Q^N=I-P^N$,
\begin{equation}
\begin{aligned}
\vert I_2\vert
&\leq 
\Vert A     (P^N  \bu - \bu^N)\Vert_{L^2}
\Vert  Q^N\bu  \Vert_{L^4}
\Vert \nabla\bu^N\Vert_{L^4}
\\
&\leq 
\frac{\nu}{6}
\Vert A     (P^N  \bu - \bu^N)\Vert_{L^2}^2
+
c(\nu)
\Vert   Q^N\bu \Vert_{H^1}^2
\Vert \bu^N\Vert_{H^2}^2
\end{aligned}
\end{equation}
and
\begin{equation}
\begin{aligned}
\vert I_3\vert
&\leq 
\Vert A     (P^N  \bu - \bu^N)\Vert_{L^2}
\Vert  \bu  \Vert_{L^\infty}
\Vert \nabla(\bu^N-\bu)\Vert_{L^2}
%\\
%&\leq 
%\frac{\nu}{6}
%\Vert A     (P^N  \bu - \bu^N)\Vert_{L^2}^2
%+
%c(\nu)
%\Vert   \bu^N-\bu \Vert_{H^1}^2
%\Vert \bu\Vert_{H^2}^2
\\
&\leq 
\frac{\nu}{6}
\Vert A     (P^N  \bu - \bu^N)\Vert_{L^2}^2
+
c(\nu)
\Vert   P^N\bu -\bu^N \Vert_{H^1}^2
\Vert \bu\Vert_{H^2}^2
\\&\quad+
c(\nu)
\Vert   Q^N\bu \Vert_{H^1}^2
\Vert \bu\Vert_{H^2}^2.
\end{aligned}
\end{equation}
Next, we also note that
\begin{equation}
\begin{aligned}
\label{vastlysimplifiedbiadvectiondiff}
\frac{1}{2} &\Big\{
 \big\langle
 A^{1/2}    
   P^N P[(\bm{\xi}_k \cdot \nabla)P(\bm{\xi}_k \cdot \nabla)\bu^N]
\,,\,   A^{1/2}     \bu^N 
\big\rangle
\\&+
 \big\langle
 A^{1/2}    
   P^N P[(\bm{\xi}_k \cdot \nabla) P(\bm{\xi}_k \cdot \nabla)\bu]
\,,\,  A^{1/2}    P^N  \bu
\big\rangle
\\& 
-  
\big\langle
A^{1/2}    P^N P [(\bm{\xi}_k\cdot\nabla ) P(\bm{\xi}_k\cdot\nabla) \bu]
\,,\,  A^{1/2}\bu^N
\big\rangle
\\&-
\big\langle
A^{1/2} P^N P [(\bm{\xi}_k\cdot\nabla) P(\bm{\xi}_k\cdot\nabla) \bu^N]
\,,\, A^{1/2}    P^N \bu
\big\rangle
\Big\}
%\\&=
%-
%\frac{1}{2} 
%\big\langle
%A^{1/2} P^N P [(\bm{\xi}_k\cdot\nabla) (\bm{\xi}_k\cdot\nabla) \bu^N]
%\,,\, A^{1/2}    (P^N \bu - \bu^N)
%\big\rangle
%\\&+
%\frac{1}{2} 
%\big\langle
%A^{1/2} P^N P [(\bm{\xi}_k\cdot\nabla) (\bm{\xi}_k\cdot\nabla) \bu]
%\,,\, A^{1/2}    ( P^N\bu -\bu^N)
%\big\rangle
%\\&=
%-
%\frac{1}{2} 
%\big\langle
%A^{1/2} P^N P [(\bm{\xi}_k\cdot\nabla) (\bm{\xi}_k\cdot\nabla) (\bu^N -\bu)]
%\,,\, A^{1/2}    (P^N \bu - \bu^N)
%\big\rangle
\\&
=
\frac{1}{2} 
\Big\{
\big\langle
A^{1/2} P^N P [(\bm{\xi}_k\cdot\nabla)P (\bm{\xi}_k\cdot\nabla) (\bu^N -\bu)]
\,,\, A^{1/2}    (\bu^N - \bu)
\big\rangle
\\&
-
\big\langle
A^{1/2} P^N P [(\bm{\xi}_k\cdot\nabla)P (\bm{\xi}_k\cdot\nabla) (\bu^N -\bu)]
\,,\, A^{1/2}    (P^N \bu - \bu)
\big\rangle
\Big\}
\\&
=
\frac{1}{2} 
\big\langle
A^{1/2} P^N P [(\bm{\xi}_k\cdot\nabla)P (\bm{\xi}_k\cdot\nabla) (\bu^N -\bu)]
\,,\, A^{1/2}    (\bu^N - \bu)
\big\rangle
\end{aligned}
\end{equation}
where we have used in the last equation, the identity $\langle P^N\bx,Q^N\by \rangle=0$ which holds for the orthogonal projection $P^N$ and its complement $Q^N=I-P^N$.
Moving on, we estimate the following three terms as follows
\begin{equation}
\begin{aligned}
\label{splitnoisedifference}
&
\frac{1}{2}\big\Vert    
   P^NP[\bg_k(\bu^N)
   -
   ((\bm{\xi}_k \cdot \nabla)\bu^N)]
\big\Vert_{H^1}^2 
+
\frac{1}{2}
\big\Vert    
   P^NP[\bg_k(\bu )
   -
   ((\bm{\xi}_k \cdot \nabla)\bu )]
\big\Vert_{H^1}^2
\\&-
\big\langle
A^{1/2}    P^N P[\bg_k(\bu) 
- ((\bm{\xi}_k\cdot \nabla)\bu)] 
\,,\, A^{1/2} P^NP\big[\bg_k(\bu^N)- ((\bm{\xi}_k \cdot\nabla )\bu^N) \big] 
\big\rangle
\\&=
\frac{1}{2}
\big\Vert    
   P^NP[\bg_k(\bu^N)
   -
   ((\bm{\xi}_k \cdot \nabla)\bu^N)]
-   
   P^NP[\bg_k(\bu )
   -
   ((\bm{\xi}_k \cdot \nabla)\bu )]
\big\Vert_{H^1}^2
\\&\leq
\big\Vert    
   P^NP[\bg_k(\bu^N)
   -
   \bg_k(\bu )]
\big\Vert_{H^1}^2
+
\big\Vert  
   P^N P[
   (\bm{\xi}_k \cdot \nabla)(\bu^N
   -
    \bu )]
\big\Vert_{H^1}^2
\end{aligned}
\end{equation}
where by the continuity of $P^NP$ and the Lipschitz property of $(\bg_k)_{k\geq1}$, it follows that
\begin{align*}
\sum_{k\geq1}
\big\Vert    
   P^NP[\bg_k(\bu^N)
   -
   \bg_k(\bu )]
\big\Vert_{H^1}^2
&\lesssim
 \Vert    
    \bu^N 
   -
    \bu  
 \Vert_{H^1}^2
\\& \lesssim
 \Vert    
    P^N\bu
   -
    \bu^N 
 \Vert_{H^1}^2
 +
\Vert    
    Q^N\bu  
 \Vert_{H^1}^2
\end{align*}
holds uniformly in $N\in \mathbb{N}$.
On the other hand, by using Lemma \ref{cor:xiAppendix} and the assumption on the transport noise in Section \ref{sec:assumptions}, we can combine \eqref{vastlysimplifiedbiadvectiondiff} with the second right-hand term in \eqref{splitnoisedifference} to obtain
\begin{equation}
\begin{aligned}
\bigg\vert 
\sum_{k\geq1}&\big\{\big\Vert  
   P^N P[
   (\bm{\xi}_k \cdot \nabla)(\bu^N
   -
    \bu )]
\big\Vert_{H^1}^2
\\&+ 
\big\langle
A^{1/2} P^N P [(\bm{\xi}_k\cdot\nabla) P(\bm{\xi}_k\cdot\nabla) (\bu^N -\bu)]
\,,\, A^{1/2}    (\bu^N - \bu)
\big\rangle\big\}
 \bigg\vert
 \\&\lesssim
   \Vert
   \bu^N-\bu 
\Vert_{H^1}^2
\lesssim
\Vert    
    P^N\bu
   -
    \bu^N
 \Vert_{H^1}^2
 +
\Vert    
    Q^N\bu  
 \Vert_{H^1}^2
\end{aligned}
\end{equation}
uniformly in $N\in \mathbb{N}$.
If we now recall \eqref{solenoidalDifference}, then it follows from the above (in)equations that
\begin{equation}
\begin{aligned}
\label{solenoidalDifference1}
\frac{1}{2} \Vert      (P^N  \bu-\bu^N)(t) \Vert^2_{H^1}
&+
\frac{\nu}{2}
\int_0^t\Vert A(P^N\bu - \bu^N)\Vert_{L^2}^2\ds
\\&\leq
\frac{1}{2} \Vert      P^N  \bu_0 -\bu^N_0 \Vert^2_{H^1}
\\&+
\int_0^t \mathcal{M}(P^N\bu \,\big\vert\, \bu^N) \dd W^k_\sigma
\\&+
c(\nu)
\int_0^t \mathcal{R}(P^N\bu \,\big\vert\, \bu^N) \ds
\end{aligned}
\end{equation}
where
\begin{equation}
\begin{aligned}
\label{realMartingale}
\mathcal{M}(P^N\bu \,\big\vert\, \bu^N)
&:=
\sum_{k\geq 1}
 \big\langle
 A^{1/2}    
   P^NP[\bg_k(\bu) - ((\bm{\xi}_k \cdot \nabla)\bu)]
\,,\,  A^{1/2}    (P^N  \bu- \bu^N ) 
\big\rangle 
\\&
-
\sum_{k\geq1}
\big\langle
A^{1/2} P^NP\big[\bg_k(\bu^N)- ((\bm{\xi}_k \cdot\nabla )\bu^N) \big] 
\,,\, A^{1/2}    (P^N  \bu- \bu^N )
\big\rangle  
\\&=
\sum_{k\geq 1}
\big\langle
 A^{1/2}    
   P((\bm{\xi}_k \cdot \nabla)(\bu^N-\bu))
\,,\,   A^{1/2}     (P^N  \bu-\bu^N) 
\big\rangle 
\\&
-
\sum_{k\geq1}
\big\langle
A^{1/2} P[ \bg_k(\bu^N)- \bg_k(\bu)]
\,,\, A^{1/2}    (P^N  \bu-\bu^N)
\big\rangle 
\end{aligned}
\end{equation}
is such that $\int_0^t\mathcal{M}(\cdot)\dd W_\sigma^k$ is a family of real-valued centred martingales and 
\begin{equation}
\begin{aligned}
\label{residualTerm}
\mathcal{R}(P^N\bu \,\big\vert\, \bu^N)
&:=
\Big(\Vert  P^N  \bu - \bu^N \Vert_{H^1}^2
+
\Vert    
    Q^N\bu  
 \Vert_{H^1}^2
 \Big)
 \\&\times
 \Big(1+\Vert \bu^N\Vert_{H^2}^2+\Vert \bu\Vert_{H^2}^2\Big)
\end{aligned}
\end{equation}
is the remainder term. Note from \eqref{bu} and \eqref{buN} that
\begin{align}
\Vert      P^N  \bu_0 -\bu^N_0 \Vert^2_{H^1}
=0.
\end{align}
In order to get rid of the martingale term in \eqref{solenoidalDifference1}, we are going to take the expectation of \eqref{solenoidalDifference1}. However, in order to also handle the nonlinear interactions due to the fluid's advection, we also introduce the a.s. strictly positive stopping time $\tau_R$  defined as 
\begin{align*}
\tau_R:=\inf \bigg\{t\in(0,T)\,:\, \int_0^t \big(\Vert\bu^N(\sigma)\Vert_{H^2}^2 + \Vert\bu(\sigma)\Vert_{H^2}^2 \big)\ds\geq R\bigg\}
\end{align*}
for some $R>0$. { Note that $\mathbb{P}[\tau_R\leq0]=0$.}
%such that $\lim_{R\rightarrow\infty}\tau_R=\mathtt{t}>0$ a.s. 
Since the martingale term is centred, it follows that
\begin{equation}
\begin{aligned}
\mathbb{E} \int_0^{t\wedge   \tau_R} \mathcal{M}(P^N\bu \,\big\vert\, \bu^N) \dd W^k_\sigma  
=
0.
\end{aligned}
\end{equation}
By collecting the various estimates above and adding the resulting inequality to \eqref{qNEst}, we obtain
\begin{equation}
\begin{aligned}
\label{solenoidalDifference1a}
\mathbb{E}&  \Big(
\Vert  (P^N  \bu-\bu^N)(t\wedge   \tau_R) 
\big\Vert_{H^1}^2
+
\Vert ( Q^N  \bu) (t\wedge   \tau_R) 
\big\Vert_{H^1}^2
\Big)
\\&+
\nu
\int_0^{t\wedge   \tau_R}\Vert A(P^N\bu - \bu^N)\Vert_{L^2}^2\ds
\\&\leq
c(\nu)
\mathbb{E}
\int_0^{t\wedge   \tau_R} 
\big(
\Vert    
    P^N\bu
   -
    \bu^N 
 \Vert_{H^1}^2
 +
 \Vert    
    Q^N\bu  
 \Vert_{H^1}^2
 \big)
 \\&\times
 \big(1+\Vert \bu^N\Vert_{H^2}^2+\Vert \bu\Vert_{H^2}^2\big)  \ds
 \\&
 { 
+c(\nu,\bu_0,K)\mathbb{E}e^{-2\mathtt{t} N}
}
\end{aligned}
\end{equation}
for  $\mathtt{t}=\delta_1\wedge \tau_R$.
{ 
Now if $\mathtt{t}=\delta_1$ as given in \eqref{qNEst}, then $\mathbb{E}e^{-2\mathtt{t} N}=e^{-2\delta_1 N}$. On the other hand, if  $\mathtt{t}=\tau_R$, then we can use the strict positivity of $\tau_R$ to infer the existence of a deterministic constant $\delta_2>0$ such that $\mathbb{P}[\tau_R\geq \delta_2]>0$. Consequently, it also follows that
\begin{align*}
    \mathbb{E}e^{-2\mathtt{t} N}\leq \mathbb{E}[e^{-2\delta_2 N}\bm{1}_{\tau_R\geq \delta_2}]\leq \mathbb{E}[e^{-2\delta_2 N}]\mathbb{P}[{\tau_R\geq \delta_2}]
    \leq  C e^{-2\delta_2 N}
\end{align*}
when $\mathtt{t}=\tau_R$.
If we now use the definition of the stopping time $\tau_R$,  then it follows from Gr\"onwall's lemma that
\begin{equation}
\begin{aligned}
\label{solenoidalDifference1b}
\mathbb{E} \Big(
\Vert  (P^N  \bu-\bu^N)(t\wedge   \tau_R) 
\big\Vert_{H^1}^2
&+
\Vert ( Q^N  \bu) (t\wedge   \tau_R) 
\big\Vert_{H^1}^2
\Big)
\\&\leq
e^{c(\nu)(T+
R
 )}
c(\nu,\bu_0,K)\,e^{-2\delta N},
\end{aligned}
\end{equation}
where $\delta=\min\{\delta_1, \delta_2\}$.
}
Our desired result thus follow due to \eqref{splitEst}.

\section{Appendix}
\label{sec:appendix}
\noindent In this section, we give the proofs of the results stated in Section \ref{sec:preparation}.
\begin{proof}[Proof of Lemma \ref{lem:h1convectiveX}]
The proof can be directly deduced from the proof of \cite[Theorem 5.3]{oliver1996mathematical} where the Gevrey space $D(e^{\varphi  A^{1/2}} : \mathbb{H}^r(\mathbb{T}^d))$ is shown to be a Banach algebra when $r>d/2$. The condition $r>d/2$ is only used to obtain \cite[(5.22)]{oliver1996mathematical}. Our desired result is however obtained from \cite[(5.21)]{oliver1996mathematical}.
See also, the original proof in \cite[Lemma 1]{ferrari1998gevrey}.
For completeness, however, we reproduce the proof below.
\begin{equation}
\begin{aligned}
\Vert A^{1/2} e^{\varphi A^{1/2}} (\bu\cdot\bv)\Vert_{L^2}^2
&=
\sum_{\bl}
\bigg\vert \sum_{\bj+\bk=\bl}\hat{\bu}_{\bj} \cdot\hat{\bv}_{\bk}\bigg\vert^2\vert\bl\vert^2 e^{2\varphi \vert \bl\vert}
\\&\leq
\sum_{\bl}
\bigg( \sum_{\bj+\bk=\bl}\vert\hat{\bu}_{\bj} \vert \,\vert\hat{\bv}_{\bk}\vert\, \vert\bl\vert e^{\varphi \vert \bl\vert}\bigg)^2
\\&\lesssim
\sum_{\bl}
\bigg( \sum_{\bj+\bk=\bl}\vert\hat{\bu}_{\bj}\vert\, \vert\bj\vert  e^{\varphi \vert \bj\vert}\vert\hat{\bv}_{\bk}\vert  e^{\varphi \vert \bk\vert}\bigg)^2
\\&+
\sum_{\bl}
\bigg( \sum_{\bj+\bk=\bl}\vert\hat{\bu}_{\bj}\vert e^{\varphi \vert \bj\vert} \vert\hat{\bv}_{\bk} \vert\,\vert\bk\vert e^{\varphi \vert \bk\vert}\bigg)^2.
\end{aligned}
\end{equation}
If we now use Young's convolution inequality followed by Jensen's inequality, we obtain,
\begin{equation}
\begin{aligned}
\Vert A^{1/2} e^{\varphi A^{1/2}} (\bu\cdot\bv)\Vert_{L^2}^2
&\lesssim
\sum_{\bj}\vert\hat{\bu}_{\bj}\vert^2 \vert\bj\vert ^2 e^{2\varphi \vert \bj\vert}
\bigg( \sum_{\bk}\vert\hat{\bv}_{\bk}\vert  e^{\varphi \vert \bk\vert}\bigg)^2
\\&+
\sum_{\bk}\vert\hat{\bv}_{\bk} \vert^2\vert\bk\vert^2 e^{2\varphi \vert \bk\vert}
\bigg( \sum_{\bj}\vert\hat{\bu}_{\bj}\vert e^{\varphi \vert \bj\vert} \bigg)^2
\\&\lesssim
\Vert A^{1/2} e^{\varphi A^{1/2}} \bu\Vert_{L^2}^2
\Vert  e^{\varphi A^{1/2}} \bv\Vert_{L^2}^2
\\&+
\Vert A^{1/2} e^{\varphi A^{1/2}} \bv\Vert_{L^2}^2
\Vert  e^{\varphi A^{1/2}} \bu\Vert_{L^2}^2.
\end{aligned}
\end{equation}
\end{proof}

{ 
\begin{proof}[Proof of Lemma \ref{cor:xiAppendix}] 
%By using \eqref{strongAssumption}, we have that
%\begin{align*}
%\big\langle
% A&^r   e^{\varphi  A^{1/2s}  }
%     ((\bm{\xi}_k \cdot \nabla) (\bm{\xi}_k \cdot \nabla)\bu )
%\, ,\,   A^r   e^{\varphi  A^{1/2s}  } \bu
%\big\rangle
%\\&=
%\big\langle
%     (\bm{\xi}_k \cdot \nabla)A^r   e^{\varphi  A^{1/2s}  } ((\bm{\xi}_k \cdot \nabla)\bu )
%\, ,\,  A^r   e^{\varphi  A^{1/2s}  } \bu
%\big\rangle
%\\
%&=
%-
%\big\langle
%     A^r   e^{\varphi  A^{1/2s}  } ((\bm{\xi}_k \cdot \nabla)\bu )
%\, ,\, (\bm{\xi}_k \cdot \nabla) A^r   e^{\varphi  A^{1/2s}  } \bu
%\big\rangle
%\\
%&=
%-
%\big\langle
%     A^r   e^{\varphi  A^{1/2s}  } ((\bm{\xi}_k \cdot \nabla)\bu )
%\, ,\,  A^r   e^{\varphi  A^{1/2s}  } ((\bm{\xi}_k \cdot \nabla)\bu )
%\big\rangle
%\\&
%=-\big\Vert
% A^r  e^{\varphi  A^{1/2s}  }
%    ((\bm{\xi}_k \cdot \nabla)\bu)
%\big\Vert_{L^2}^2
%\end{align*}
To simplify notations, for $\bj,\bk \in \mathbb{Z}^d$,
we let  $j:=\vert\mathbf{j}\vert$ and $k:=\vert\mathbf{k}\vert$ and note that for $\bn \in  \mathbb{Z}^d$,
\begin{align*}
\frac{1}{(2\pi)^d}\int_{\mathbb{T}^d}e^{i\bn\cdot\bx}\dx
= 
  \begin{cases}
    0      & \quad \text{if } \bn\neq\bm{0},\\
    1  & \quad \text{if } \bn=\bm{0}.
  \end{cases}
\end{align*}
Combining this information with $\langle f,g\rangle=\frac{1}{(2\pi)^d}\int_{\mathbb{T}^d}f(x)\overline{g}(x)\dx$, we obtain for
\begin{align*}
\bm{\xi}_k(\bx)= \sum_{\bj\in  \mathbb{Z}^d}\hat{\bm{\xi}_k}_{\bj} e^{i\bj\cdot\bx}\qquad  \text{and} \qquad 
\bu(\bx)= \sum_{\bk\in  \mathbb{Z}^d}\hat{\bu}_{\bk} e^{i\bk\cdot\bx}
\end{align*}
that 
\begin{align*}
\big\langle&
 A^r   e^{\varphi  A^{1/2s}  }
    (\bm{\xi}_k \cdot \nabla  (\bm{\xi}_k \cdot \nabla)\bu )
\, ,\,   A^r   e^{\varphi  A^{1/2s}  } \bu
\big\rangle
\\&=
\frac{1}{(2\pi)^d}
\Big\langle 
  \sum_{\bj}  \hat{\bm{\xi}_k}_{\bj}  e^{i\bj\cdot\bx} \cdot    \sum_{\bj}i\bj\,  \hat{\bm{\xi}_k}_{\bj}  e^{i\bj\cdot\bx} \cdot \sum_{\bk}i\bk\,\hat{\bu}_\bk e^{i\bk\cdot\bx}
\,,\,  
 \sum_{\bk}k^{4r}e^{2\varphi  k^{1/s}  }\hat{\bu}_{\bk} e^{i\bk\cdot\bx} 
\Big\rangle 
\\&= 
\sum_{\bn}
\sum_{2(\bj+\bk)=\bn}
   (\hat{\bm{\xi}_k}_{\bj} \cdot  i\bj) (\hat{\bm{\xi}_k}_{\bj} \cdot  i\bk)k^{4r} e^{2\varphi  k^{1/s}} \vert \hat{\bu}_\bk\vert^2\frac{1}{(2\pi)^d}\int_{\mathbb{T}^d} e^{i\bn\cdot\bx} \dx 
  \\&=  
\sum_{2(\bj+\bk)=\bm{0}}k^{4r} e^{2\varphi  k^{1/s}  }
  (\hat{\bm{\xi}_k}_{\bj} \cdot  i\bj) (\hat{\bm{\xi}_k}_{\bj} \cdot  i\bk) \vert \hat{\bu}_\bk\vert^2
    \\&\overset{\bj=-\bk}{=} - 
\sum_{ \bk }
k^{4r} e^{2\varphi  k^{1/s}  } \vert\hat{\bm{\xi}_k}_{-\bk} \cdot  i\bk\vert^2 \vert \hat{\bu}_\bk\vert^2
  \\&\overset{\overline{\hat{\bff}}_{\bk}= \hat{\bff}_{-\bk}}{=} - 
\sum_{ \bk }
k^{4r} e^{2\varphi  k^{1/s}  } \vert\overline{\hat{\bm{\xi}_k}}_{\bk} \cdot  i\bk\vert^2 \vert \hat{\bu}_\bk\vert^2
 \\&=- 
\big\Vert
 A^r  e^{\varphi  A^{1/2s}  }
    ((\bm{\xi}_k \cdot \nabla)\bu)
\big\Vert_{L^2}^2
\end{align*}
\end{proof}
}

{ 
Next, we present the proof of Lemma \ref{lem:comparison} for completeness. The proof essentially follow the same argument as \cite[Lemma 5.1]{glatt2009strong} with the obvious modification with respect to the exponential weight
\begin{proof}[Proof of Lemma \ref{lem:comparison}]
Given that \eqref{expoZero} holds, we can find an increasing family $(N_l)_{l\in\mathbb{N}_0}$ such that
\begin{align*}
\mathbb{E} \Vert  \bu^{N_{l+1}} -\bu^{N_l}  \Vert_{\mathcal{E}(\mathcal{T}^M_{N_{l+1}N_l})} \leq 2^{-2l}.
\end{align*}
With this, we define the stopping time
\begin{align*}
\tau_l:=\inf \bigg\{  t\in [0,T]\,:\,
\Vert \bu^{N_l} \Vert_{\mathcal{E}(t)}
>
\Vert   e^{\varphi(0)A^{1/2s} }\bu^{N_l}_0 \Vert_{H^r} 
+ M-1+2^{-l}
\bigg\} \wedge T
\end{align*}
so that $\tau_l\wedge\tau_{l+1}\in \mathcal{T}^M_{N_{l+1}N_l}$ and thus, by Chebyshev's inequality,
\begin{align*}
\mathbb{P} \big(\Vert  \bu^{N_{l+1}} -\bu^{N_l}  \Vert_{\mathcal{E}(\tau_l\wedge\tau_{l+1})} 
\geq 2^{-(l+2)}\big)
&\leq 2^{l+2}
\mathbb{E} \Vert  \bu^{N_{l+1}} -\bu^{N_l}  \Vert_{\mathcal{E}(\tau_l\wedge\tau_{l+1})}  
\\&
\leq  2^{l+2}2^{-2l}
\\&=2^{2-l}.
\end{align*}
Given the finiteness of this probability, by setting
\begin{align}
\label{omegaN}
\Omega_p:=\bigcap_{l=p}^\infty\big\{\Vert  \bu^{N_{l+1}} -\bu^{N_l}  \Vert_{\mathcal{E}(\tau_l\wedge\tau_{l+1})} 
< 2^{-(l+2)}\big\},
\end{align}
we can use Borel--Cantelli lemma to conclude that 
\begin{equation}
\begin{aligned}
\label{borelCantelli}
0&=\mathbb{P}\bigg(\bigcap_{p=1}^\infty\bigcup_{l=p}^\infty\big\{\Vert  \bu^{N_{l+1}} -\bu^{N_l}  \Vert_{\mathcal{E}(\tau_l\wedge\tau_{l+1})} 
\geq2^{-(l+2)}\big\}
\bigg) 
\\&=\mathbb{P}\bigg(\bigcap_{p=1}^\infty\Omega_p
^c\bigg)
\\&=1-\mathbb{P}\bigg(\bigcup_{p=1}^\infty\Omega_p
\bigg)
\end{aligned}
\end{equation}
and thus, the set $\tilde{\Omega}=\bigcup_{p=1}^\infty\Omega_p$ has full measure.
\\
Now consider the set $\{\tau_l<\tau_{l+1}\}\cap \Omega_p$ where $l\geq p$, so that $\tau_l<T$. By  using the continuity of $\Vert \bu^{N_l} \Vert_{\mathcal{E}(t)}$ in $t$ (in the sense of Remark \ref{rem:weakContinuity}), it follows that
\begin{align}
\label{a1}
\Vert \bu^{N_l} \Vert_{\mathcal{E}(\tau_l)}
=
\Vert   e^{\varphi(0)A^{1/2s} }\bu^{N_l}_0 \Vert_{H^r} 
+ M-1+2^{-l}.
\end{align}
On the other hand, we obtain on $\Omega_p$ (c.f. \eqref{omegaN}) that
\begin{align}
&\Vert \bu^{N_l}  \Vert_{\mathcal{E}(\tau_l\wedge\tau_{l+1})} 
-
\Vert  \bu^{N_{l+1}} \Vert_{\mathcal{E}(\tau_l\wedge\tau_{l+1})} 
< 2^{-(l+2)},
\label{a2}
\\&
\Vert   e^{\varphi(0)A^{1/2s} }\bu^{N_{l+1}}_0 \Vert_{H^r} 
-
\Vert   e^{\varphi(0)A^{1/2s} }\bu^{N_l}_0 \Vert_{H^r} 
< 2^{-(l+2)}.
\label{a3}
\end{align}
If we now combine \eqref{a1}-\eqref{a3}, then we obtain on the set $\{\tau_l<\tau_{l+1}\}\cap \Omega_p$ that
\begin{equation}
\begin{aligned}
\label{a4}
\Vert  \bu^{N_{l+1}} \Vert_{\mathcal{E}(\tau_l\wedge\tau_{l+1})}
&>
\Vert \bu^{N_l}  \Vert_{\mathcal{E}(\tau_l\wedge\tau_{l+1})} 
- 2^{-(l+2)}
\\&
=
\Vert \bu^{N_l}  \Vert_{\mathcal{E}(\tau_l)} 
- 2^{-(l+2)}
\\&
=
\Vert   e^{\varphi(0)A^{1/2s} }\bu^{N_l}_0 \Vert_{H^r} 
+ M-1+2^{-l}
- 2^{-(l+2)}
\\&
>
%\Vert   e^{\varphi(0)A^{1/2s} }\bu^{N_{l+1}}_0 \Vert_{H^r} 
%+ M-1+2^{-l}
%- 2^{-(l+2)}- 2^{-(l+2)}
%\\&
%=
\Vert   e^{\varphi(0)A^{1/2s} }\bu^{N_{l+1}}_0 \Vert_{H^r} 
+ M-1+2^{-(l+1)}.
\end{aligned}
\end{equation}
Similarly, on $\{\tau_l<\tau_{l+1}\}\cap \Omega_p$, we can also show that 
\begin{align}
\label{a5}
\Vert  \bu^{N_{l+1}} \Vert_{\mathcal{E}(\tau_l\wedge\tau_{l+1})}
\leq
\Vert   e^{\varphi(0)A^{1/2s} }\bu^{N_{l+1}}_0 \Vert_{H^r} 
+ M-1+2^{-(l+1)}
\end{align}
by noticing that the reverse inequalities
\begin{align*}
&\Vert  \bu^{N_{l+1}} \Vert_{\mathcal{E}(\tau_l\wedge\tau_{l+1})} 
-
\Vert \bu^{N_l}  \Vert_{\mathcal{E}(\tau_l\wedge\tau_{l+1})} 
< 2^{-(l+2)}, 
\\&
\Vert   e^{\varphi(0)A^{1/2s} }\bu^{N_l}_0 \Vert_{H^r} 
-
\Vert   e^{\varphi(0)A^{1/2s} }\bu^{N_{l+1}}_0 \Vert_{H^r} 
< 2^{-(l+2)}. 
\end{align*}
also holds on $\Omega_p$.
The two estimates \eqref{a4} and \eqref{a5} implies that $\{\tau_l<\tau_{l+1}\}\cap \Omega_p=\emptyset$ which further implies that for every $l\geq p$ and $\omega\in \Omega_p$, we have that
\begin{align}
\label{decreasingStoppingTime}
\tau_{l+1}(\omega)\leq \tau_l(\omega).
\end{align}
Thus, by combining \eqref{borelCantelli} with the decreasing sequence \eqref{decreasingStoppingTime}, we obtain that
\begin{align*}
\tau_l\rightarrow\tau \qquad a.s.
\end{align*}
with 
\begin{align*}
\tau\leq T\qquad a.s.
\end{align*}
following due to the definition of $\tau_l$. To show \eqref{strictlyPositiveStopTime}, it remains to show that $\tau>0$ a.s. or equivalently that $\mathbb{P}(\tau=0)=0$. For this, we fix $\varepsilon>0$ with $T>\varepsilon>0$ so that
\begin{align*}
\liminf_{l\rightarrow\infty} \{\tau_l<\varepsilon \}
&=
\{\tau<\varepsilon\}
\\&
\subset
\big\{ \Vert \bu^{N_l} \Vert_{\mathcal{E}(\tau_l\wedge \varepsilon)}
=
\Vert   e^{\varphi(0)A^{1/2s} }\bu^{N_l}_0 \Vert_{H^r} 
+ M-1+2^{-l} \big\}
\\&
\subset
\big\{ \Vert \bu^{N_l} \Vert_{\mathcal{E}(\tau_l\wedge \varepsilon)}
>
\Vert   e^{\varphi(0)A^{1/2s} }\bu^{N_l}_0 \Vert_{H^r} 
+ M-1 \big\}.
\end{align*}
Therefore,
\begin{align*}
\mathbb{P}( \tau=0)
&=
\lim_{\varepsilon\rightarrow0}\mathbb{P}( \tau<\varepsilon)
\\&=
\lim_{\varepsilon\rightarrow0}\mathbb{P}\Big(\liminf_{l\rightarrow\infty} \{\tau_l<\varepsilon \}
\Big)
\\
&\leq
\lim_{\varepsilon\rightarrow0}
\liminf_{l\rightarrow\infty} \mathbb{P}(\tau_l<\varepsilon  )
\\
&\leq
\lim_{\varepsilon\rightarrow0}
\limsup_{l\rightarrow\infty} \mathbb{P}(\tau_l<\varepsilon )
\\
&\leq
\lim_{\varepsilon\rightarrow0}
\sup_{l \in\mathbb{N}_0} \mathbb{P}\Big(\Vert \bu^{N_l} \Vert_{\mathcal{E}(\tau_l\wedge \varepsilon)}
>
\Vert   e^{\varphi(0)A^{1/2s} }\bu^{N_l}_0 \Vert_{H^r} 
+ M-1
\Big)
\\&\overset{\eqref{probZero}}{=}0.
\end{align*}
This completes the proof of \eqref{strictlyPositiveStopTime}.
\\
Now, to show \eqref{cauchyZero}, we observe from \eqref{decreasingStoppingTime} that  for any $\omega\in \tilde{\Omega}$, we can choose $p=p(\omega)$ so that 
whenever $l\geq p$, we have that $\omega\in \Omega_p$ and
\begin{align}
\label{decreasingStoppingTime1}
\tau(\omega)\leq
\tau_{l+1}(\omega)\leq \tau_l(\omega).
\end{align}
Thus, it follows from \eqref{omegaN} that
\begin{align*}
\Vert  \bu^{N_{l+1}} -\bu^{N_l}  \Vert_{\mathcal{E}(\tau(\omega))} 
\leq
\Vert  \bu^{N_{l+1}} -\bu^{N_l}  \Vert_{\mathcal{E}(\tau_l(\omega)\wedge\tau_{l+1}(\omega))} 
< 2^{-(l+2)}.
\end{align*}
We can, therefore, find a process $\bu(\cdot)=\bu(\cdot\wedge \tau)\in \mathcal{E}(\tau)$ such that \eqref{cauchyZero} holds.
\\
Finally, to show \eqref{cauchyZeroM}, we use \eqref{a1}, \eqref{a5} and \eqref{decreasingStoppingTime1} and obtain
\begin{align*}
\bm{1}_{\Omega_l}\Vert \bu^{N_l}  \Vert_{\mathcal{E}(\tau )} 
&\leq
 \bm{1}_{\Omega_l}\Vert  \bu^{N_{l+1}} \Vert_{\mathcal{E}(\tau )}
+ 2^{-(l+2)}
\\
&\leq 
\Vert   e^{\varphi(0)A^{1/2s} }\bu^{N_{l+1}}_0 \Vert_{H^r} 
+ M-1+2^{-(l+1)}+ 2^{-(l+2)}
\\
&\leq 
\Vert   e^{\varphi(0)A^{1/2s} }\bu^{N_{l+1}}_0 \Vert_{H^r} 
+ M
\\
&\leq 
\sup_N
\Vert   e^{\varphi(0)A^{1/2s} }\bu^{N_{l}}_0 \Vert_{H^r} 
+ M
\end{align*}
almost surely. 
\end{proof}
}

\noindent Finally, we present a useful identify for final dimensional orthogonal projection. In particular, for a Hilbert space $V$, we note that since the identity
\begin{align*}
\langle (P^N-P^n)f^N\,,\,P^n(f^N-f^n)\rangle_V
=
\langle P^n(P^N-P^n)f^N\,,\,P^n(f^N-f^n)\rangle_V=0
\end{align*}
holds for all $f^N,f^n\in V$ and $N\geq n$, it follows that
\begin{equation}
\begin{aligned}
\label{projectionIdentity}
\Vert P^Nf^N-P^nf^n \Vert_V^2
&=\Vert (P^N-P^n)f^N+P^n(f^N-f^n) \Vert_V^2
\\
&= \Vert (P^N-P^n)f^N  \Vert_V^2
+
\Vert P^n(f^N-f^n) \Vert_V^2
\\
&=\Vert Q^nP^Nf^N  \Vert_V^2
+
\Vert P^n(f^N-f^n) \Vert_V^2.
\end{aligned}
\end{equation}

\section*{Acknowledgements}
\subsection{Funding} This work has been supported by the European Research Council (ERC) Synergy grant STUOD-DLV-856408.
\subsection{Author Contribution} All authors wrote and reviewed the manuscript.
\subsection{Conflict of Interest} The authors declare that they have no conflict of interest.
\subsection{Data Availability Statement} Data sharing is not applicable to this article as no datasets were generated or
analysed during the current study.

%\noindent This work has been supported by the European Research Council (ERC) Synergy grant STUOD-DLV-856408.

%\bibliographystyle{spmpsci}
%\bibliography{myBibliography}

\end{document}